\documentclass[a4paper, reqno]{amsart}
\usepackage{color}
\usepackage{cite}
\usepackage{mathpazo}
\usepackage{mdwlist}
\usepackage{mathrsfs}

\newtheorem{theorem}{Theorem}[section]
\newtheorem{lemma}[theorem]{Lemma}

\theoremstyle{remark}

\newtheorem{remark}[theorem]{Remark}
\newtheorem{definition}[theorem]{Definition}

\newcommand{\coloremphasize}{\relax}

\numberwithin{equation}{section}

\def\cC{{\mathfrak{C}}}
\def\N{{\mathbb{N}}}
\def\Z{{\mathbb{Z}}}
\def\Rl{{\mathbb{R}}}
\def\Cx{{\mathbb{C}}}
\def\sH{{\mathcal{H}}}
\def\tr{{\mathrm{Tr}\,}}

\def\mL{{\mathcal{L}}}

\def\sH{{\mathbb{H}}}
\def\W{{{W}}}
\def\Pl{{{P}}}

\begin{document}

\title{Spectral shift function of higher order}

\author[Potapov]{Denis Potapov$^{*}$}
\address{School of Mathematics and Statistics, University of New South Wales, Kensington, NSW 2052, Australia}
\email{d.potapov@unsw.edu.au}

\author[Skripka]{Anna Skripka$^{**}$}
\address{Department of Mathematics and Statistics, MSC01 1115, 1 University of New Mexico, Albuquerque, NM  87131}
\email{skripkaan@gmail.com}

\author[Sukochev]{Fedor Sukochev$^{*}$} \address{School of Mathematics
  and Statistics, University of New South Wales, Kensington, NSW 2052,
  Australia} \email{f.sukochev@unsw.edu.au}

\thanks{\footnotesize $^{*}$Research supported in part by ARC}
\thanks{\footnotesize $^{**}$Research supported in part by NSF grant
  DMS-0900870 and by AWM-NSF Mentoring Travel Grant}



\date{\today}

\maketitle

\begin{abstract}
  This paper resolves affirmatively Koplienko's conjecture of 1984 on
  existence of higher order spectral shift measures. Moreover, the
  paper establishes absolute continuity of these measures and, thus,
  existence of the higher order spectral shift functions. A spectral
  shift function of order~$n \in \N$ is the function~$\eta_n =
  \eta_{n, H, V}$ such that
  \begin{equation}
    \label{nTraceFormula}
    \tr \left( f(H + V)-\sum_{k = 0}^{n-1} \frac 1{k!}\, \frac {d^k}{dt^k}
      \left[ f(H + tV) \right]\Bigr|_{t = 0} \right) = \int_\Rl f^{(n)}
    (t)\, \eta_n  (t)\, dt,
  \end{equation}
  for every sufficiently smooth function~$f$, where~$H$ is a
  self-adjoint operator defined in a separable Hilbert space $\sH$ and
  $V$ is a self-adjoint operator in the $n$-th Schatten-von Neumann
  ideal $S^n$. Existence and summability of $\eta_1$ and $\eta_2$ were
  established by Krein in 1953 and Koplienko in 1984, respectively,
  whereas for~$n > 2$ the problem was unresolved. We show that~$\eta_{n, H, V}$ exists, integrable,
  and $$ \left\| \eta_n \right\|_{L^1(\Rl)} \leq c_n \, \left\| V
  \right\|_{S^n}^n, $$ for some constant~$c_n$ depending only on~$n
  \in \N$. Our results for $\eta_n$ rely on estimates for multiple operator integrals obtained in this paper. Our method also applies to the general semi-finite von Neumann algebra setting of the perturbation theory.
\end{abstract}

\section{Introduction}
\label{sec:intro}

The first order spectral shift function $\eta_1$ originated from
Lifshits' work on theoretical physics \cite{Lifshits}. The
mathematical theory of this object was founded by Krein in a series of
papers, starting with \cite{Krein}. The spectral shift function has
become a fundamental object in perturbation theory. It can also be
recognized as the scattering phase \cite{BirmanKrein} and the spectral
flow in a non-commutative geometry setting \cite{ACS}. The original
spectral shift function applies only in the case of trace class
perturbations (or trace class differences of the resolvents). The
modified second order spectral shift function for Hilbert-Schmidt
perturbations was introduced by Koplienko in \cite{Kop84}. In 1984,
Koplienko also conjectured existence of the higher order spectral
shift measures $\nu_n$, $n>2$, for the perturbation $V$ in $S^n$ such
that
\begin{equation}
\label{nTrnu}
\tr \left( f(H + V)-\sum_{k = 0}^{n-1} \frac 1{k!}\, \frac {d^k}{dt^k}
  \left[ f(H + tV) \right]\Bigr|_{t = 0} \right) = \int_\Rl f^{(n)}
(t)\,d\nu_n(t).
\end{equation}
In \cite{Kop84}, Koplienko offered a proof which contained an
unremovable gap (see \cite{ds} for details and \cite{BirmanY,GPS} for
additional comments and historical information). In this paper, we
prove existence and absolute continuity of $\nu_n$. Note that the
density of $\nu_n$ is the spectral shift function $\eta_n$.

Before stating the main result, we need to fix some notation.
Throughout the paper, $n$ denotes a natural number.  Let~$C^n$ denote
the space of all $n$~times continuously differentiable complex-valued
functions on $\Rl$, $C^n_c$ the subclass of~$C^n$ of compactly
supported functions. We also set~$C := C^0$ and $C_c := C^0_c$ and let
$C_b$ denote the subclass of continuous bounded functions.  Let~$\W_n
\subseteq C^n$ be the class of all functions~$f \in C^n$ such that the
Fourier transform of~$f^{(n)}$ is integrable.  For~$f \in \W_n$, we
set $$ \left\| f \right\|_{\W_n} := \int_{\Rl} \left|
  \widehat{f^{(n)}} (s) \right|\, ds.$$ The class $\W_n$ includes the
functions $f$ for which $f^{(n)}$ and $f^{(n+1)}$ are in $L^2(\Rl)$
\cite[Lemma 7]{PS-Crelle}. In particular, $C^{n+1}_c\subset\W_n$.

The first and second order spectral shift functions have been
introduced in the von Neumann algebra setting as well.  Let~$M
\subseteq B(\sH)$ be a semifinite von Neumann algebra and~$\tau$ a
normal faithful semifinite trace on~$M$.  Let~$L^\alpha$ denote the
noncommutative $L^\alpha$-space with respect to~$(M,\tau)$ and
$\mL^\alpha$ the $\tau$-Schatten-von Neumann ideal $L^\alpha\cap M$
(see, e.g., \cite{AzCaDoSu2009, PiXu2003} and references cited therein
for basic definitions and facts). The existence of $\eta_1$ and
$\eta_2$ satisfying~(\ref{nTraceFormula}), where the standard trace
$\tr$ is replaced with $\tau$ and $V$ is taken from $\mL^1$ and
$\mL^2$, respectively, is due to \cite{ADS06,CareyPincus} and
\cite{ds}.

Our result on the spectral shift functions is stated below.

\begin{theorem}
\label{SSFtheorem}
Let $n\in\N$. Let~$H$ be a self-adjoint operator affiliated with~$M$
and let~$V$ be a self-adjoint operator in $\mL^n$. Denote $H_t: = H +
tV$, $0 \leq t \leq 1$, and let $f\in \cap_{k=0}^n\W_k$. Denote
   \begin{equation}
   \label{remdef}\Delta_{n, f} (H, V):= f\left( H+V \right) - \sum_{k =
      0}^{n-1} \frac 1{k!} \, \frac {d^k}{dt^k} \left[ f\left( H_t
      \right) \right] \biggr|_{t = 0}.
  \end{equation}
Then~$\Delta_{n, f} (H, V) \in \mL^1$ and
there is a unique function~$\eta_n = \eta_{n, H, V} \in L^1(\Rl)$ depending only on $n,H,V$ such that
    \begin{equation}
      \label{SSFfunctionRel}
      \tau \left( \Delta_{n, f} (H, V) \right) = \int_{\Rl}f^{(n)} (t)\, \eta_n (t)\, dt
    \end{equation}
    and $$ \left\| \eta_n \right\|_1 \leq \, c_n \, \left\| V\right\|_n^n. $$
\end{theorem}

The existence of $\nu_1$ via double operator integration techniques
was established in \cite{BirmanS}. The absolute continuity of $\nu_1$,
or, equivalently, the existence of $\eta_1$, was established by Krein
years before the development of double operator integration by
analytic function and approximation theory techniques. Up to date,
there is no double operator integral proof of the absolute continuity
of $\eta_1$. The proof of the existence of $\eta_2$ is due to
Koplienko; it is a more delicate application of double operator
integration. The existence of $\eta_2$ is also established in the
present paper.

Our proof of Theorem \ref{SSFtheorem} (see Section
\ref{sec:main_result}) for $n\geq 2$ is based on a new powerful
estimate for multiple operator integrals (see Theorem
\ref{MainTheoremCor} in Section \ref{sec:main_result} and Theorem
\ref{MainTheoremExt} and Remark \ref{MTER} in Section
\ref{sec:doi-fn}), which extends the advances of \cite{PS-Lipschitz}
for the first order case to the higher order ones.  The spectral shift
function of order $n$, with $n>2$, was constructed explicitly in
\cite{ds,unbdds,moissf} under the restrictive assumptions $V\in \mL^2$
when $M=B(\sH)$ and $V\in \mL^2$, $n=3$ when $M$ is a general
semi-finite von Neumann algebra. The higher order $\eta_n$ can be
expressed recursively via the lower order $\eta_k$, $k\leq n$.

The paper is organized as follows. The proof of Theorem
\ref{SSFtheorem} is given in Section \ref{sec:main_result}. The
technically involved auxiliary results are proved in several steps in
subsequent sections. A novel approach to multiple operator integrals
suitable to the purposes of this paper is discussed in Section
\ref{sec:mod-new}.

\subsection*{Acknowledgements} We are thankful to the referees for
useful comments which improved the exposition of the paper.

\section{Proof of Theorem \ref{SSFtheorem}}
\label{sec:main_result}

Our proof is based on the following powerful estimate for the
remainder of the Taylor-type approximation in~(\ref{SSFfunctionRel}).

\begin{theorem}
  \label{MainTheoremCor}
  Let $n\in\N$. Let~$H$ be a self-adjoint operator affiliated with~$M$
  and let~$V$ be a self-adjoint operator in $\mL^n$. Denote~$H_t: = H
  + tV$, $0 \leq t \leq 1$, and let~$f\in \cap_{k=0}^n\W_k$. Then
  $\frac{d^n}{dt^n}\left[f\left(H_t\right)\right]\in \mL^1$ and
  $\Delta_{n,f}(H,V)\in \mL^1$ and there are constants $c_n$ and
  $c_n'$ (depending only on $n$) such that the estimates
\begin{equation}
\label{TextPrincipalEstimateD}
\left|\tau\left(\frac {d^n}{dt^{n}}\left[f\left(H_t\right)\right]\right)\right|\leq c_n\, \left\|
  f^{(n)} \right\|_\infty\, \left\| V\right\|_n^n
\end{equation}
and
\begin{equation}
\label{TextPrincipalEstimate}
\left|\tau \left(\Delta_{n, f} (H, V) \right) \right| \leq \, c_n'\, \left\|
        f^{(n)} \right\|_\infty\, \left\| V\right\|_n^n
\end{equation} hold.
In addition, the mapping $V\mapsto \tau\left(\Delta_{n, f} (H,
   V)\right)$ is continuous on $\mL^n$ uniformly with respect to $f$
 in $\left(\cap_{k=0}^n\W_k\right)\cap B$, where $B$ is the unit ball
 of $C^{n}$ taken with the seminorm $f\mapsto \|f^{(n)}\|_\infty$.
\end{theorem}

{\coloremphasize
\begin{remark}
  Theorem~\ref{MainTheoremCor} significantly improves earlier known
  estimates of this type.  For instance, it follows
  from~\cite{PellerMult} that
  \begin{equation*}
    \left|\tau\left(\frac {d^n}{dt^{n}}\left[f
          \left(H_t\right)\right]\right)\right|\leq c_n\, \left\|
      f \right\|_{B_{\infty1}^n(\Rl)}\, \left\| V\right\|_n^n,
  \end{equation*} where $f$ is in the Besov class~$B_{\infty1}^n(\Rl)$.
  However, the norm $\left\| f \right\|_{B_{\infty1}^n(\Rl)}$ is
  greater, than the norm $\|f^{(n)}\|_\infty$, so this estimate is
  weaker than~(\ref{TextPrincipalEstimateD}).
\end{remark}
}

The proof of Theorem~\ref{MainTheoremCor} is given in the following
sections.  Assuming that Theorem~\ref{MainTheoremCor} holds, we prove
Theorem~\ref{SSFtheorem}.

\begin{proof}[Proof of Theorem~\ref{SSFtheorem}]

  By the inequality~(\ref{TextPrincipalEstimate}) applied to all $f\in
  C_c^{n+1}$, the Riesz representation theorem for a bounded linear
  functional on the space of continuous functions on a compact set
  ensures existence of a unique finite real-valued measure $\nu_n$ on
  $\Rl$, with
\begin{equation}\label{estimate}
\left\|\nu_n\right\|\leq c_n\, \left\| V \right\|_n^n,
\end{equation} and such that
  \begin{equation}
    \label{ssmeasure}
    \tau \left( \Delta_{n, f} (H, V) \right) = \int_{\Rl} f^{(n)}
    (t)\, d\nu_n (t).
  \end{equation}
  {\coloremphasize We extend~(\ref{ssmeasure}) from the subset
    $C_c^{n+1}$ 
    to the whole $\cap_{k=0}^n\W_k$ by approximations.
  }

  To finish the proof, we need to demonstrate the absolute continuity
  of~$\nu_n$ for $n\geq 2$ (absolute continuity of~$\nu_1$ was
  established in \cite{ADS06,CareyPincus,Krein}).

  Firstly, we assume that $V\in \mL^1\subseteq\mL^n$ and, therefore,
  Theorem \ref{MainTheoremCor} guarantees that $\Delta_{n,f}(H,V)\in
  \mL^1$. Then, for every $f\in C_c^n$, integration by parts gives
  \begin{equation}
    \label{parts1}
    \tau\left(\Delta_{n-1,f}(H,V)\right)=\int_\Rl
    f^{(n-1)}(t)\,d\nu_{n-1}(t) =-\int_\Rl
    f^{(n)}(t)\nu_{n-1}((-\infty,t))\,dt.
  \end{equation}
  By~(\ref{TextPrincipalEstimateD}), we have
  \begin{equation*}
    \left|\tau\left(\frac {d^{n-1}}{dt^{n-1}} \left[ f(H_t)
        \right]\right)\right| \leq c_{n-1}\, \left\|
      f^{(n-1)}\right\|_\infty\, \left\| V\right\|_{n-1}^{n-1},\ \
    f\in C_c^n,
  \end{equation*}
  and thus the Riesz representation theorem implies the existence of a
  unique finite measure~$\mu_{n-1}$ on~$\Rl$ such that
  \begin{align}
    \label{parts2}
    \nonumber \frac{1}{(n-1)!}\tau\left(\frac {d^{n-1}}{dt^{n-1}}
      \left[ f(H_t) \right]\right)&=\int_\Rl
    f^{(n-1)}(t)\,d\mu_{n-1}(t)\\&=-\int_\Rl
    f^{(n)}(t)\mu_{n-1}((-\infty,t))\,dt.
  \end{align}
  By combining~(\ref{parts1}) and~(\ref{parts2}) in~(\ref{remdef}), we
  obtain
  \begin{equation*}
    \tau\left(\Delta_{n,f}(H,V)\right)
    =\int_\Rl
    f^{(n)}(t)\big[\mu_{n-1}((-\infty,t))-\nu_{n-1}((-\infty,t))\big]\,dt,
  \end{equation*}
  which along with~(\ref{ssmeasure}) implies that $\nu_n$ is
  absolutely continuous and its density equals \[\eta_n(t)= \mu_{n-1}
  ((-\infty,t)) - \nu_{n-1}((-\infty,t)).\] Due to~(\ref{estimate}),
  we have that $\|\eta_n\|_1\leq c_n \|V\|_n^n$. Thus, the existence
  of the spectral shift function of order $n$ is proved for $V \in
  \mL^1$.

  To prove the existence of $\eta_n$ in the case of a general $V\in
  \mL^n$, we choose a sequence of operators $\{V_k\}_k \subseteq
  \mL^1$ such that $ \lim_{k\rightarrow\infty} \|V-V_k\|_n = 0$ (see,
  e.g., \cite{ChilinS}). We now show that the sequence of the
  (integrable) spectral shift functions $\{\eta_{n,H,V_k}\}_k$ is
  Cauchy in $L^1(\Rl)$. First, by duality, we obtain
\begin{align*}
  &\int_\Rl\left|\eta_{n,H,V_j}(t)-\eta_{n,H,V_k}(t)\right|\,dt\\
  &\quad=\sup_{f\in C^{n+1}_c,\; \|f^{(n)}\|_\infty\leq
    1}\left|\int_\Rl\left(\eta_{n,H,V_j}(t)
      -\eta_{n,H,V_k}(t)\right)f^{(n)}(t)\,dt\right|.
\end{align*}
  By~(\ref{ssmeasure}),
  \begin{align*}
    &\left|\int_\Rl\left(\eta_{n,H,V_j}(t)
        -\eta_{n,H,V_k}(t)\right)f^{(n)}(t)\,dt\right| =
    \left|\tau\left(\Delta_{n,f}(H,V_j) -
        \Delta_{n,f}(H,V_k)\right)\right|.
  \end{align*}
  The uniform continuity of
  $V\mapsto\tau\left(\Delta_{n,f}(H,V)\right)$ (see Theorem
  \ref{MainTheoremCor}) implies
  \[\lim_{j,k\rightarrow\infty}\int_\Rl\left|\eta_{n,H,V_j}(t)
    -\eta_{n,H,V_k}(t)\right|\,dt=0.\] Thus, the sequence
  $\{\eta_{n,H,V_k}\}_k$ converges to an integrable function, which we
  denote by $\eta_{n, H, V}$. Since $L^1$-norms of the functions
  $\eta_{n,H,V_k}$, $k\in\N$, are uniformly bounded by $c_n\|V\|^n_n$,
  we obtain $$ \|\eta_{n, H, V} \|_1\leq c_n\|V\|^n_n. $$ By passing
  to the limit in the representation
  \[\tau \left( \Delta_{n, f} (H, V_k) \right) = \int_{\Rl}
  f^{(n)} (t)\, \eta_{n,H,V_k}(t)\, dt\] as $V_k\rightarrow V$, we
  obtain that $\eta_{n,H,V}$ satisfies the trace
  formula~(\ref{SSFfunctionRel}) and, thus, it is the spectral shift
  function of order~$n$ corresponding to the general perturbation
  $V\in\mL^n$.

\end{proof}

\begin{remark}
  The earlier mentioned explicit representation of \cite{ds} for
  $\eta_n$ along with the summability of $\eta_n$ implies
  \begin{equation}
    \label{etaIntegral}
    \int_\Rl\eta_n(t)\,dt=\tau(V^n)/n!
  \end{equation}
  for $V$ in the Hilbert-Schmidt class \cite{unbdds} and thus, by
  approximations, for $V$ in the $n$-th Schatten von Neumann ideal. We
  remark that the equality~(\ref{etaIntegral}) is a routine
  calculation in the case of a bounded operator $H$ (for more details
  see, e.g., \cite[Lemma 3.5]{ds}).  It follows from the trace formula
  (\ref{SSFfunctionRel}) applied to a function $f\in W_n$, which
  coincides with the polynomial $t^n$ on the spectra of all operators
  involved in~(\ref{SSFfunctionRel}). The property~(\ref{etaIntegral})
  in the case of an unbounded operator $H$ can also be obtained by
  approximations from the bounded case. The equality
  (\ref{etaIntegral}) for $n=1$ was obtained in \cite{Krein,ADS06} and
  for $n=2$ in \cite{PellerKo}.
\end{remark}

\section{Multilinear transformations $T_\phi$}
\label{sec:mod-new}

Let~$H$ be a self-adjoint linear operator affiliated with~$M$ and
let~$dE_\lambda$, $\lambda \in \Rl$ be the corresponding spectral
measure.  We set~$E_{l,m} = E \left[ \frac lm, \frac {l + 1}m
\right)$, for every~$m \in \N$ and~$l \in \Z$.  The symbols~$H$,
$dE_\lambda$ and~$E_{l, m}$ will keep their meaning for the rest of
the {\coloremphasize paper}.

\begin{definition}
\label{nmoi}
Let~$n \in \N$ and let $1 \leq \alpha_j \leq \infty$, with $1 \leq j
\leq n$, be such that $0 \leq \frac 1{\alpha_1} +\ldots + \frac
1{\alpha_n} \leq 1$. Let ~$x_j \in L^{\alpha_j}$ and denote $\tilde
x:=(x_1, \ldots, x_n )$.  Fix a bounded Borel function~$\phi: \Rl^{n +
  1} \mapsto \Cx$.  Suppose that for every~$m \in \N$, the
series\footnote{understood as\; $\lim_{N\rightarrow\infty}
  \sum_{|l_j|\leq N,\; 0\leq j\leq n}\phi \left( \frac {l_0}m, \ldots,
    \frac {l_n}m \right) \, E_{l_0, m} x_1 E_{l_1, m} x_2 \cdot \ldots
  \cdot x_n E_{l_n, m}$}
\begin{equation*}
  S_{\phi, m} (\tilde x) := \sum_{l_0, \ldots,
    l_n \in \Z} \phi \left( \frac {l_0}m, \ldots, \frac
    {l_n}m \right) \, E_{l_0, m} x_1 E_{l_1, m} x_2 \cdot \ldots \cdot
  x_n E_{l_n, m}
\end{equation*}
converges in the norm of~$L^\alpha$, where
$\frac1\alpha=\frac1{\alpha_1} + \ldots + \frac 1{\alpha_n}$ and $$
\tilde x \mapsto S_{\phi, m} (\tilde x),\ \ m \in \N, $$ is a sequence
of bounded multilinear operators~$L^{\alpha_1} \times \ldots \times
L^{\alpha_n} \mapsto L^\alpha$.  If the sequence of
operators~$\left\{S_{\phi,m}\right\}_{m \geq 1}$ converges strongly to
some multilinear operator~$T_\phi$, then, according to the
Banach-Steinhaus theorem, $\left\{S_{\phi,m}\right\}_{m \geq 1}$ is
uniformly bounded and the operator~$T_\phi$ is also bounded.
If~$T_\phi$ exists as above and bounded, we
shall be simply saying that~$T_\phi$ is bounded on~$L^{\alpha_1}
\times \ldots \times L^{\alpha_n}$ for brevity.
\end{definition}

Throughout the paper, we shall frequently use the following algebraic
properties of the mapping~$\phi \mapsto T_\phi$.  The proof is immediate
from the definition above.

\begin{lemma}
  \label{MOI-algebra}
  Let~$1 \leq\alpha_j \leq \infty$, for $1\leq j \leq n$, be such that
  $0 \leq \frac{1}{\alpha_1}+\ldots+\frac1{\alpha_n}\leq 1$.
  Let~$x_j \in L^{\alpha_j}$, $1\leq j \leq n$.

\begin{enumerate}
\item Let~$\phi: \Rl^{n + 1} \mapsto \Cx$ be bounded Borel and
  let~$T_\phi$ be bounded on~$L^{\alpha_1} \times \ldots \times
  L^{\alpha_n}$.  If $$ \bar \phi(\lambda_0, \lambda_1, \ldots,
  \lambda_n) := \overline{\phi (\lambda_n, \lambda_{n-1}, \ldots,
    \lambda_0)}, $$ then~$T_{\bar \phi}$ is bounded on~$L^{\alpha_1}
  \times \ldots \times L^{\alpha_n}$ and $$ \left\| T_{\phi} \right\|
  = \left\| T_{\bar \phi} \right\|. $$ \label{MOI-A-involution}
\item Assume, in addition, that $1 \leq\alpha_0\leq\infty$ and
  $\frac{1}{\alpha_0}+\ldots+\frac1{\alpha_n}=1$.  Let~$\phi: \Rl^{n +
    1} \mapsto \Cx$ be a bounded Borel function. Assume that $T_\phi$
  exists and is bounded on~$L^{\alpha_1} \times \ldots \times
  L^{\alpha_n}$.  Define $$ \phi^*(\lambda_{n}, \lambda_0, \ldots,
  \lambda_{n-1}) := \phi \left( \lambda_0, \ldots, \lambda_{n-1},
    \lambda_n \right). $$ If~$T_{\phi^*}$ exists and is bounded
  on~$L^{\alpha_0} \times \ldots \times L^{\alpha_{n-1}}$, then
  \footnote{Note that we do not imply boundedness of~$T_{\phi^*}$ from
    that of~$T_{\phi}$.  This is due to the fact that strong operator
    topology is not well compatible with duality, i.e., there is an
    example of a sequence of operators converging strongly such that
    the sequence of dual operators does not converge.}
  $$ \tau\left( x_0 T_\phi (x_1, \ldots, x_n) \right) = \tau \left(
    T_{\phi^*} \left( x_0, \ldots , x_{n-1} \right) x_n
  \right). $$ \label{MOI-A-duality}
  \item Let~$\phi_1: \Rl^{k + 1} \mapsto \Cx$ and~$\phi_2 : \Rl^{n - k
      + 1} \mapsto \Cx$ be bounded Borel functions such that the
    operators~$T_{\phi_1}$ and~$T_{\phi_2}$ exist and are bounded
    on~$L^{\alpha_1} \times \ldots \times L^{\alpha_k}$ and~$L^{\alpha_{k +
        1}} \times \ldots \times L^{\alpha_n}$, respectively.  If $$
    \psi(\lambda_0, \ldots, \lambda_n) := \phi_1 \left( \lambda_0,
      \ldots, \lambda_k \right) \cdot \phi_2 \left( \lambda_k, \ldots,
      \lambda_n \right), $$ then
    the operator~$T_{\psi}$ exists and is bounded on~$L^{\alpha_1}
    \times \ldots \times L^{\alpha_n}$ and $$ T_{\psi} \left( x_1,
      \ldots, x_n \right) = T_{\phi_1} \left( x_1, \ldots, x_k \right)
    \cdot T_{\phi_2} \left( x_{k + 1}, \ldots, x_n
    \right). $$ \label{MOI-A-product}
  \item Let~$\phi_1: \Rl^{k + 1} \mapsto \Cx$ and~$\phi_2 : \Rl^{n - k
      + 2} \mapsto \Cx$ be bounded Borel functions such
    that~$T_{\phi_1}$ and~$T_{\phi_2}$ exist and are bounded
    on~$L^{\alpha_1} \times \ldots \times L^{\alpha_k}$
    and~$L^{\alpha} \times L^{\alpha_{k+1}} \times \ldots \times
    L^{\alpha_n}$, respectively, where
    $\frac1\alpha=\frac{1}{\alpha_1}+\ldots+\frac1{\alpha_k}$. If $$
    \psi (\lambda_0, \ldots, \lambda_n): = \phi_1 \left( \lambda_0,
      \ldots, \lambda_k \right) \cdot \phi_2 \left(\lambda_0,
      \lambda_{k}, \ldots, \lambda_n \right), $$ then
    the operator~$T_{\psi}$ exists and is bounded
    on~$L^{\alpha_1}\times\ldots\times L^{\alpha_n}$ and
    $$T_{\psi} \left( x_1, \ldots, x_n \right) = T_{\phi_2} \left(
      T_{\phi_1}(x_1, \ldots, x_k), x_{k + 1}, \ldots, x_n
    \right). $$ \label{MOI-A-composition}
  \end{enumerate}
\end{lemma}

The next lemma shows that in dealing with the operators~$T_\phi$ it is always
sufficient to consider compactly supported functions~$\phi$.

\begin{lemma}
  \label{CompactSupportTphi}
  Let~$\phi: \Rl^{n+1} \mapsto \Cx$ be a bounded Borel function and
  let $1 \leq \alpha_j < \infty$, with $1\leq j\leq n$, be such that
  $0\leq\frac1{\alpha_1}+\ldots+\frac1{\alpha_n} \leq 1$. If
  \begin{equation*}
    \phi_k  \left( \tilde \lambda
    \right) :=
    \begin{cases}
      \phi \left( \tilde \lambda \right), & \text{if~$\left| \lambda_j
        \right| \leq k$, for every~$0 \leq j \leq n$;} \\ 0, &
      \text{otherwise,}
    \end{cases}
  \end{equation*}
  where~$\tilde \lambda = \left( \lambda_0, \ldots, \lambda_n \right)
  \in \Rl^{n + 1}$, and if the sequence of the
  operators~$\{T_{\phi_k}\}_{k=1}^\infty$ exists and is uniformly
  bounded on~$L^{\alpha_1} \times \ldots \times L^{\alpha_n}$, then
  the operator~$T_\phi$ exists and is bounded (with the same norm
  estimate) on~$L^{\alpha_1} \times \ldots \times L^{\alpha_n}$.
\end{lemma}

\begin{proof}
  Let~$F_k = E([-k, k])$ and let~$L_k^\alpha = F_k L^\alpha F_k
  \subseteq L^\alpha$.  Observe that if the collection~$\tilde x =
  \left( x_1, \ldots, x_n \right)$ is such that~$x_j \in
  L^{\alpha_j}_k$, then $$ y_m : = S_{\phi, m} (\tilde x) = S_{\phi_k,
    m} (\tilde x). $$ Since the operator $T_{\phi_k}$ exists and is
  bounded, the sequence $\{y_m\}_{m\geq 1}$ converges in $L^\alpha$
  for every $k\geq 1$. Hence, the operator $T_{\phi}$ is well defined
  and bounded on the subspace $\bigcup_{k \geq
    1}L^{\alpha_1}_k\times\ldots\times L^{\alpha_n}_k$. Since the set
  $\bigcup_{k \geq 1} L^{\alpha_1}_k \times \ldots \times
  L^{\alpha_n}_k$ is norm dense in~$L^{\alpha_1} \times \ldots \times
  L^{\alpha_n}$, the operator $T_{\phi}$ is well defined on
  $L^{\alpha_1} \times \ldots \times L^{\alpha_n}$.
\end{proof}

\begin{remark}
\label{phiHS}
In the special case $n=1$, the multilinear operator $T_\phi$ becomes a
linear operator on $L^\alpha$. For a large class of functions $\phi$, the operator $T_\phi$ coincides with the
double operator integral $\hat{T}_\phi$ \cite{PSW2002} (see below and also see
references cited in \cite{BirmanS, PSW2002}). In particular, it is
known that on the Hilbert space $L^2$, the norm of the operator
$\hat T_\phi$ is bounded by $\|\phi\|_\infty$, for every bounded
Borel~$\phi$.  Note that this nice Hilbert space behaviour is
exclusive to the case~$n = 1$.  For~$n \geq 2$ it seems there is no
combination of exponents~$\alpha_1, \ldots, \alpha_n$ such
that~$T_\phi$ is bounded for every bounded Borel~$\phi$.
\end{remark}

We next introduce the subclass of functions~$\phi$ for which it is
relatively simple to show that the operator~$T_\phi$ is bounded, though the
functions themselves have rather complex structure.

Let~$\cC_n$ be the class of functions~$\phi: \Rl^{n + 1} \mapsto \Cx$
admitting the representation
\begin{equation}
  \label{Crep}
  \phi (\lambda_0, \ldots, \lambda_n) =
  \int_\Omega \prod_{j = 0}^n a_j (\lambda_j, s)\, d\mu(s),
\end{equation}
for some finite measure space~$\left( \Omega, \mu \right)$ and bounded
continuous functions $$ a_{j} \left( \cdot, s \right) : \Rl \mapsto
\Cx $$ for which there is a growing sequence of measurable
subsets~$\left\{\Omega_k \right\}_{k \geq 1}$, with $\Omega_k \subseteq
\Omega$ and~$\cup_{k \geq 1} \Omega_k = \Omega$ such that the
families $$ \left\{a_{j} (\cdot, s) \right\}_{s \in \Omega_k},\ \ 0
\leq j \leq n, $$ are uniformly bounded and uniformly equicontinuous.  The
class~$\cC_n$ has the norm $$ \left\| \phi \right\|_{\cC_n} = \inf
\int_\Omega \prod_{j = 0}^n \left\| a_j(\cdot, s) \right\|_\infty \,
d\left| \mu \right|(s), $$ where the infimum is taken over all
possible representations~(\ref{Crep}).

The following lemma demonstrates that for $\phi\in\cC_n$, Definition
\ref{nmoi} coincides with the one in \cite{PellerMult} (see also \cite{AzCaDoSu2009}). We point out that the definition in \cite{PellerMult} also applies to a wider set of functions.

\begin{lemma}
  \label{ClassNewDef}
 Let $1\leq \alpha_j \leq \infty$, with $1\leq j\leq n$, be such that $0 \leq \frac 1{\alpha_1} + \ldots + \frac 1{\alpha_n} \leq 1$. For every~$\phi \in \cC_n$, the operator $T_\phi$ exists and is bounded
  on~$L^{\alpha_1} \times \ldots \times L^{\alpha_n}$, with
  $$ \left\| T_\phi \right\| \leq \left\|\phi \right\|_{\cC_n}. $$
  Moreover, given the decomposition~(\ref{Crep}) of the function
  $\phi$, the operator~$T_\phi$ can be represented as the Bochner
  integral
  \begin{equation}
    \label{ClassNewDefRep}
    T_\phi (x_1, x_2, \ldots, x_n)
    = \int_{\Omega} a_0(H, s)\, x_1 \, a_1 (H, s)\, x_2 \cdot \ldots\cdot x_n a_n (H, s)\, d \mu(s).
  \end{equation}
\end{lemma}

\begin{proof}[Proof of Lemma~\ref{ClassNewDef}]
Due to the choice of functions $a_j$, the operator
 \begin{equation}
    \label{ClassNewDefRep'}
    \hat{T}_\phi(x_1, x_2, \ldots, x_n):= \int_{\Omega} a_0(H, s)\,
    x_1 \, a_1 (H, s)\, x_2 \cdot \ldots\cdot x_n a_n (H, s)\, d
    \mu(s)
  \end{equation}
  is bounded on $L^{\alpha_1} \times \ldots \times L^{\alpha_n}$,
  with $$ \left\| \hat{T}_\phi \right\| \leq \left\|
    \phi\right\|_{\cC_n}.$$

  Let us show that the multiple operator
  integral~$\hat{T}_\phi$ coincides with the multilinear transformation given by Definition \ref{nmoi}.  Let $H_m:=\sum_{l \in \Z}
  \frac lm E_{l, m}$.  Clearly,
  \begin{equation}
    \label{ClassNewDefTempI}
    \left\| H - H_m
    \right\| \leq \frac 1m.
  \end{equation}
  By analogy with~(\ref{ClassNewDefRep'}), we set
  \begin{equation*}
    \hat{T}_{\phi, m} (x_1, x_2, \ldots, x_n):
    = \int_{\Omega} a_0(H_m, s)\, x_1 \, a_1 (H_m, s)\, x_2 \cdot \ldots
    \cdot x_n a_n (H_m, s)\, d \mu(s).
  \end{equation*}
  It can be seen that the operator~$\hat{T}_{\phi, m}$ coincides
  with~$S_{m,\phi}$ in Definition \ref{nmoi}, i.e.,
  \begin{equation}
  \label{c1}
  S_{\phi,m}\left( \tilde x \right) =
  \hat{T}_{\phi, m} \left( \tilde x \right), \quad \tilde x = (x_1,
  \ldots, x_n).
  \end{equation}
  Thus, to finish the proof, it suffices to show that $$ \lim_{m
    \rightarrow \infty} \left\| \hat T_{\phi, m} (\tilde x) -
    \hat{T}_{\phi} \left( \tilde x \right) \right\|_\alpha = 0,\quad
  \tilde x = \left( x_1, \ldots, x_n\right).  $$ To this end,
  fix~$\tilde x = (x_1, \ldots, x_n)$, with $\prod_{j=1}^n\|x_j\|_{\alpha_j}\leq 1$, fix~$\epsilon > 0$, and fix a
  number~$k_\epsilon \in \N$ such that
  \begin{equation}
    \label{ClassNewDefTemp}
    \int_{\Omega \setminus \Omega_{k_\epsilon}} \prod_{j = 0}^n \left\|
      a_j(\cdot, s) \right\|_\infty\, d \left| \mu \right|(s) < \epsilon.
  \end{equation}
  We next set
  \begin{multline*}
    \hat{y}_\epsilon := \int_{\Omega_{k_\epsilon}} a_0(H, s) x_1 \cdot
    \ldots \cdot x_n a_n(H, s)\, d\mu (s) \quad \text{and}\\
    \hat{y}_{\epsilon, m}: = \int_{\Omega_{k_\epsilon}} a_0(H_m, s)
    x_1 \cdot \ldots \cdot x_n a_n(H_m, s)\, d\mu (s).
  \end{multline*}
  The estimate~(\ref{ClassNewDefTemp}) implies
  \begin{equation}
    \label{c2}
    \left\| \hat{T}_{\phi}
      (\tilde x) - \hat{y}_\epsilon \right\|_\alpha < \epsilon \quad
    \text{and}\quad \left\| \hat{T}_{\phi, m}(\tilde x) -
      \hat{y}_{\epsilon, m} \right\|_\alpha <\epsilon.
  \end{equation}
  Since the family of functions $\{a_j(\cdot,
  s)\}_{s\in\Omega_{k_\epsilon}}$, $0\leq j\leq n$, is uniformly bounded and uniformly
  equicontinuous, we derive from~(\ref{ClassNewDefTempI}) existence of
  $m_\epsilon \in \N$ such that
  \begin{equation}
    \label{c3}
    \left\|\hat{y}_\epsilon-\hat{y}_{\epsilon, m}\right\|_\alpha < \epsilon,\quad \text{for any } m > m_\epsilon.
  \end{equation}
  Combining~(\ref{c1}), (\ref{c2}), and~(\ref{c3}) implies that for
  every~$\epsilon > 0$, there exists $m_\epsilon \in \N$ such that for
  every $m > m_\epsilon$,
  \begin{align*}
    \left\| \hat{T}_\phi(\tilde x) - \hat T_{\phi, m} (\tilde x)
    \right\|_\alpha \leq \left\| \hat{T}_\phi (\tilde x) -
      \hat{y}_\epsilon \right\|_\alpha + \left\| \hat{y}_\epsilon -
      \hat{y}_{\epsilon, m} \right\|_\alpha + \left\|
      \hat{y}_{\epsilon, m} - \hat T_{\phi, m} (\tilde x)
    \right\|_\alpha < 3\epsilon,
  \end{align*}
  completing the proof.
\end{proof}

Recall that for $f\in W_n$, the operator derivative
$\frac{d^n}{dt^n}[f(H+tV)]$ is given by the integral $T_{f^{[n]}}$
\cite{PellerMult} (see also \cite{AzCaDoSu2009}) where $f^{[n]}$ is the divided
difference. The divided difference of the zeroth order~$f^{[0]}$ is
the function~$f$ itself. Let~$\lambda_0, \lambda_1, \ldots \in \Rl$
and let~$f \in C^n$. The divided difference $f^{[n]}$ of order~$n$ is
defined recursively by
\begin{align*}
  f^{[n]} \left( \lambda_0, \lambda_1, \tilde \lambda \right) =
  \begin{cases}\frac
    { f^{[n-1]} (\lambda_0, \tilde \lambda) - f^{[n-1]}(\lambda_1,
      \tilde \lambda)}{\lambda_0 - \lambda_1}, & \text{if~$\lambda_0
      \neq \lambda_1$}, \\ \frac {d}{d\lambda_1} f^{[n-1]} (\lambda_1,
    \tilde \lambda), & \text{if~$\lambda_0=\lambda_1$},
  \end{cases}
\end{align*}
where~$\tilde\lambda = \left(\lambda_2, \ldots, \lambda_n \right) \in
\Rl^{n-1}$.  It is clear that if~$f \in C^n$, then the divided
difference~$f^{[n]}$ is a continuous function on~$\Rl^{n + 1}$. If~$f
\in \W_n$, then~$f^{[n]}\in\cC_n$ \cite[Lemma 2.3]{AzCaDoSu2009}.

\section{Demonstration of the method.}
\label{sec:doi-case}

{\coloremphasize The proof of Theorem~\ref{MainTheoremCor} is rather
  technical, so we choose the following approach. In the present
  section, we firstly demonstrate the principal techniques of the
  proof in some simple cases.  In our demonstration, we shall focus on
  the estimate~(\ref{TextPrincipalEstimateD}) alone.  The second
  estimate~(\ref{TextPrincipalEstimate}) follows
  from~(\ref{TextPrincipalEstimateD}) via (a higher version of) the
  Fundamental Theorem of the Calculus.  We shall show this in the main
  proof of Theorem~\ref{MainTheoremCor} in Section~\ref{sec:doi-fn}
  (see~(\ref{SSFRemainderTemp}) below).

  We first explain the connection between the
  estimate~(\ref{TextPrincipalEstimateD}) and estimates for multiple
  operator integrals.

\subsection*{The case~$n = 1$.}

In this case, the estimate~(\ref{TextPrincipalEstimateD}) is
straightforward.  Indeed, from the elementary identity $$ \tau\left(
  \frac d{dt} \left[ f(H_t) \right] \right) = \tau \left( f'(H_t)\, V
\right), $$ the proof of~(\ref{TextPrincipalEstimateD}) is
immediate $$ \left| \tau \left( \frac d{dt} \left[ f(H_t) \right]
  \right) \right| \leq \, \left\| f' \right\|_\infty\, \left\| V
\right\|_1. $$

\subsection*{The case~$n = 2$.}

For the second derivative~$\frac {d^2}{dt^2} \left[ f(H_t) \right]$, we
employ the following representation
\begin{equation}
  \label{SSFFrechetII}
\frac {d^2}{dt^2} \left[ f(H_t) \right] = 2 T_{t, f^{[2]}} \left( V, V
\right),
\end{equation}
where~$T_{t, f^{{2}}}$ is the triple operator integral
associated with the operator~$H_t$ and the divided
difference~$f^{[2]}$.  The representation above has been established
by several authors, we refer the reader
to~\cite[Theorem~5.7]{AzCaDoSu2009} (see also~(\ref{SSFFrechetTemp})
below).

To make the demonstration simpler, we assume that~$H_t$ has pure
integral point spectrum, i.e., if~$dE_\lambda$, $\lambda \in \Rl$, is
the spectral measure of~$H_t$, then $$ E(B) = \sum_{l \in B\cap\Z}
E_l,\ \ \text{$B \subseteq \Rl$ is Borel,} $$ where~$E = \left\{E_l
\right\}_{l \in \Z}$ is a family of pairwise orthogonal projections
and~$\sum_{l \in \Z} E_l = 1$.  In fact, we shall later show (see
Lemma~\ref{ToDiscreteMeasure}) that the case of an arbitrary
operator~$H_t$ can always be reduced, via approximation, to the case
of spectral measure with integral concentration points.

We keep the point~$t \in [0, 1]$ fixed, so we do not show the
dependence of the spectral measure~$dE_\lambda$ on~$t$ in our
notations.

Given that~$H_t$ has pure integral point spectrum, the representation
for the second derivative above takes the following simpler form $$
\frac {d^2}{dt^2} \left[ f(H_t) \right] = \sum_{l, m, k \in \Z}
f^{[2]} \left( l, m, k \right)\, E_{l} V E_m V E_k. $$ Furthermore,
applying the trace~$\tau$ to the identity above and coupling the
projections~$E_l$ and~$E_k$ on the right hand side, we arrive at $$
\tau \left( \frac {d^2}{dt^2} \left[ f(H_t) \right] \right) = \sum_{m,
  k \in \Z} f^{[2]}(k, m, k)\, \tau \left( V E_m V E_k \right) = \tau
\left( V\, T_{t, \phi_2} \left( V \right) \right), $$ where we set $$
T_{t, \phi_2}(V) = \sum_{m, k \in \Z} \phi_2(m, k)\, E_m V E_k \ \
\text{and}\ \ \phi_2(x, y) = f^{[2]}(x, y, x). $$ In other words,
application of the trace has reduced the triple operator integral
formula to the double operator integral formula.  Thus, the
estimate~(\ref{TextPrincipalEstimateD}) would follow from $$ \left|
  \tau (V\, T_{t, \phi_2}(V)) \right| \leq \left\| f''
\right\|_\infty\, \left\| V \right\|_2^2, $$ or, via H\"{o}lder's
inequality, from $$ \left\| T_{t, \phi_2} (V) \right\|_2 \leq \left\|
  f'' \right\|_\infty\, \left\| V \right\|_2. $$ However, the latter
follows from the observation that $$ \left\| \phi_2 \right\|_\infty
\leq \left\| f^{[2]} \right\|_\infty \leq \left\| f''
\right\|_\infty $$ and that the well-known result that

\begin{theorem}
  \label{LiiEstimateTjm}
  For every Borel bounded function~$\phi(x, y)$, the double operator
  integral~$T_\phi$ is bounded on~$L^2$ and $$ \left\| T_\phi \right\|
  \leq \left\| \phi \right\|_\infty. $$
\end{theorem}

Thus, the desired estimate~(\ref{TextPrincipalEstimateD}) is
established for~$n = 2$.  The scheme above was first presented by
L.S.~Koplienko in his paper~\cite{Kop84}.

\subsection*{The case~$n > 2$.}

For the higher derivative~$\frac {d^n}{dt^n} \left[ f(H_t) \right]$,
we use higher version of the representation~(\ref{SSFFrechetII}), that
is (see~\cite[Theorem~5.7]{AzCaDoSu2009}), $$ \frac {d^n}{dt^n} \left[
  f(H_t) \right] = n! \, T_{t, f^{[n]}} \Bigl( \underbrace{V, \ldots,
  V}_{\text{$n$-times}} \Bigr). $$ By repeating the steps we used in
the case~$n = 2$, i.e., by applying the trace to this identity and by
coupling the outer projections on the right hand side, the
estimate~(\ref{TextPrincipalEstimateD}) can be similarly reduced to
showing the following inequality $$ \biggl\| T_{\phi_n}
\Bigl(\underbrace{V, \ldots, V}_{\text{$(n-1)$-times}}
\Bigr)\biggr\|_{\frac n{n-1}} \leq c_n\, \left\|
  f^{(n)}\right\|_\infty\, \left\| V \right\|_n^{n-1}, $$ where $$
\phi_n(x_0, x_1, \ldots, x_{n-1}) = f^{[n]}(x_0, x_1, \ldots, x_{n-1},
x_0). $$ The latter is the principal technical ingredient of our
approach in this paper allowing resolution of the conjecture of
L.S.~Koplienko.

In fact, we shall show that the estimate above holds for the wide
class of functions which we call polynomial integral momenta, and for
every non-trivial $n$-tuples of spaces~$(L^{\alpha_1}, \ldots,
L^{\alpha_n})$.  This result is given in Theorem~\ref{MainTheoremExt}
below.

\subsection*{Theorem~\ref{MainTheoremExt} for the first order divided
  difference}

The proof of Theorem~\ref{MainTheoremExt} is highly technical.  Before
we proceed with the detailed argument, we shall explain a key idea
using the simplest non-trivial example of polynomial integral momentum:
the first order divided difference~$f^{[1]}$.  In the latter case,
Theorem~\ref{MainTheoremExt} takes the following simpler form.

\begin{theorem}
  \label{PSLipschitzAlt}
  If~$f \in C^1$ and $\|f'\|_\infty<\infty$, then the
  operator~$T_{f^{[1]}}$ is bounded on every~$L^\alpha$, with $1 <
  \alpha < \infty$, and $$ \left\| T_{f^{[1]}} \right\| \leq c_\alpha
  \, \left\| f' \right\|_\infty, $$ where the constant~$c_\alpha$
  depends only on $\alpha$.
\end{theorem}

Theorem~\ref{PSLipschitzAlt} was established
in~\cite[Theorem~2]{PS-Lipschitz}, but the method
of~\cite{PS-Lipschitz} does not allow extension to the higher
dimensions.  Here we give a new proof suitable for our
present purpose.

}

\begin{remark}
  The result of Theorem \ref{PSLipschitzAlt} cannot be extended to the
  cases $\alpha=1,\infty$; see \cite{F1,F2,F3} for counterexamples.
\end{remark}



Recall that the spectrum of~$H$ is concentrated at integral points.
That is, the operator~$T_{f^{[1]}}$ associated with $H$ is given by
the multiple sum $$ T_{f^{[1]}}(x)=\sum_{l,m\in\Z}f^{[1]}(l,m)\, E_{l}
x E_{m}.$$ We shall also consider~$f^{[1]}$ only on compact subsets
(as in Lemma~\ref{CompactSupportTphi}); however, we shall not reflect
this in our notations.  The latter enables to replace double operator
integrals with finite sums.  Note also that the estimates we present
below do not depend on the support of~$f^{[1]}$.

The proof of Theorem~\ref{PSLipschitzAlt} is based on the following lemmas.

\begin{lemma}[{\cite[Lemma~6]{PS-Lipschitz}}]
  \label{DecompositionLemma}
  There is a function~$g: \Rl \mapsto \Cx$ such that
$$ \int_{\Rl}\left| s \right|^n \left| g(s) \right| \, ds < + \infty,\ \ n \geq 0,$$
and such that, for every~$\mu \geq \lambda > 0$,
$$ \frac\lambda \mu = \int_{\Rl} g(s)\, \lambda^{is} \mu^{-is}\, ds. $$
\end{lemma}


\begin{lemma}[{\cite[Lemma~5]{PS-Lipschitz}}]
  \label{MarcCorl}
  Let~$x \in L^\alpha$, with $1 < \alpha < \infty$. Then $$ x_s = \sum_{l < m} (m - l)^{is} \, E_l
  x E_m,\quad s \in \Rl,$$ is well defined and there is a constant~$c_\alpha > 0$ such that
  $$ \left\| x_s  \right\|_\alpha \leq \, c_\alpha \, \left( 1 + \left| s \right|
  \right)\, \left\| x \right\|_\alpha. $$
\end{lemma}

The principal step toward Theorem~\ref{PSLipschitzAlt} is the following lemma.

\begin{lemma}
  \label{ImprovedDOI-base-rec}
  Let~$2 < \alpha, \beta < \infty$ be such that~$2^{-1} = \alpha^{-1} +
  \beta^{-1}$.  Then there is a constant~$c_\alpha > 0$ such that, for
  every~$f \in C^1$ with~$\left\| f' \right\|_\infty \leq 1$, $$
  \left\| T_{f^{[1]}} \right\|_\alpha \leq c_\alpha \left( 1 + \left\|
      T_{f^{[1]}} \right\|_\beta \right), $$ where~$\left\|
    T_{f^{[1]}} \right\|_\alpha$ is the norm of the operator $T_{f^{[1]}}:
  L^\alpha\mapsto L^\alpha$.
\end{lemma}

Throughout the text, we agree that the constant symbols~$c_\alpha$ are
allowed to vary from line to line, or even within a line.

\begin{proof}[Proof of Lemma~\ref{ImprovedDOI-base-rec}]
  Let~$x \in L^\alpha$ and let~$y \in L^{\alpha'}$ where~$\alpha'$ is
  the conjugate exponent, i.e., $\alpha^{-1} + \alpha'^{-1} = 1$.  We
  shall prove the estimate $$\left|\tau\left(y T_{f^{[1]}}(x)\right)
  \right|\leq\,c_\alpha\left(1+\left\|T_{f^{[1]}}\right\|_\beta
  \right)\,\left\|x\right\|_\alpha\,\left\|y\right\|_{\alpha'},$$ for
  some constant~$c_\alpha > 0$, which immediately implies the claim of
  the lemma.

  Let us fix~$x \in L^\alpha$ and~$y \in L^{\alpha'}$.  Without loss
  of generality, we may assume that~$\left\| x \right\|_\alpha =
  \left\| y \right\|_{\alpha'} = 1$. The triangular truncation is a
  bounded linear operator on~$L^\alpha$, $1< \alpha < \infty$ (see,
  e.g., \cite{DDPS}).  Thus, we may further assume that the
  operator~$x$ is upper-triangular and~$y$ is lower-triangular with
  respect to the family~$\left\{E_l \right\}_{l \in \Z}$.\footnote{An
    element~$x \in M$ is called {\it upper-triangular\/} with respect
    to a family of pairwise orthogonal projections~$\left\{E_l
    \right\}_{l \in \Z}$ if and only if~$E_l x E_m = 0$ for every~$l>
    m$; it is called {\it lower-triangular\/} if and only if~$x^*$ is
    upper-triangular.}  We also may assume that~$x$ is off-diagonal
  as, {\coloremphasize by standard techniques, $T_{f^{[1]}}$ is
    bounded on the diagonal subspace of~$L^\alpha$ (details can be
    found in \cite{PS-Lipschitz}).}

  We can assume that $y$ has $\tau$-finite left and right supports
  because the class of lower-triangular operators with $\tau$-finite
  supports is norm dense in the lower-triangular part of
  $L^{\alpha'}$.  Let us fix~$\epsilon > 0$. Since
  $\frac1{\alpha'}=\frac12+\frac1\beta$, there is a factorization
  $y=ab$, where $a \in L^2$ and $b \in L^\beta$ are lower-triangular
  and
  \begin{equation}
    \label{PLfactorization}
    1 \leq \left\| a \right\|_2 \, \left\| b \right\|_\beta \leq
    1 + \epsilon.
  \end{equation}
  Such factorization always exists due to
  \cite[Theorem~8.3]{PiXu2003}.

For every element~$z \in M$, we set~$z_{lm} := E_l z  E_m$ for brevity.
Since $x$ is upper triangular and $y$ is lower triangular,
\begin{equation*}
    \tau \left( y T_{f^{[1]}}(x)\right) = \tau \left( a b T_{f^{[1]}}(x)
    \right) =  \sum_{l \leq k \leq m \atop l \neq m} f^{[1]}(l, m) \tau
    \left( a_{mk}\, b_{kl}\, x_{lm} \right).
\end{equation*}
Observing the straightforward decomposition {\coloremphasize (in the
  proof of Theorem~\ref{MainTheoremExt} this simple decomposition is
  replaced by Lemma~\ref{PhiMrep})}
  \begin{equation}
    \label{PhiMrepTriv}
    f^{[1]}(l, m) = \frac
    {l - k}{l - m} \, f^{[1]} (l, k) + \frac {k
      - m}{l - m} \, f^{[1]}(k, m),\ \ l \leq k
    \leq m,\ l \neq m,
  \end{equation}
  we obtain
  \begin{multline*}
    \tau\left(yT_{f^{[1]}}(x)\right)=\sum_{l<k\leq
      m}\frac{l-k}{l-m}\,\tau\left(a_{m k}
      \left(T_{f^{[1]}}(b)\right)_{k l}x_{l m}\right) \\
    +\sum_{l\leq k <
      m}\frac{k-m}{l-m}\,\tau\left(\left(T_{f^{[1]}}(a)\right)_{m k}
      b_{k l} x_{l m} \right) =: S_1 + S_2.
  \end{multline*}
  We shall estimate the term~$S_1$. The estimate for the term~$S_2$
  can be obtained similarly. Observe that
  Lemma~\ref{DecompositionLemma} yields the representation $$ \frac
  {l- k}{l - m} = \int_{\Rl} g(s) \, \left( k - l \right)^{is} \left(
    m - l \right)^{-is}\, ds,\ \ l \leq k < m, $$ where~$g: \Rl
  \mapsto \Cx$ is such that
  \begin{equation}
    \label{Gfunction}
    \int_{\Rl} \left| s \right|^n \, \left| g(s)  \right| \, ds < +
    \infty,\ \ n \geq 0.
  \end{equation}
  Thus, if we set (as in Lemma~\ref{MarcCorl}) $$ x_s := \sum_{l < m}
  \left( m - l \right)^{is} \, x_{l m} \ \ \text{and}\ \
  \left(T_{f^{[1]}}(b)\right)_s := \sum_{l < k} \left( k - l
  \right)^{is} \, \left(T_{f^{[1]}}(b)\right)_{k l},\ \ s \in \Rl, $$
  then, by Lemma \ref{MOI-algebra} (\ref{MOI-A-product}), $$ S_1 =
  \int_{\Rl} g(s) \, \tau \left( a \left(T_{f^{[1]}}(b)\right)_s
    x_{-s} \right)\, ds. $$ Subsequent application of the
  noncommutative H\"older inequality, Lemma~\ref{MarcCorl} to both
  $\left(T_{f^{[1]}}(b)\right)_s$ and~$x_s$, and the estimate
  (\ref{PLfactorization}) implies
  \begin{multline*}
    \left| \tau \left( a \left(T_{f^{[1]}}(b)\right)_s x_{-s} \right)
    \right| \leq \left\| a \right\|_2 \, \left\|
      \left(T_{f^{[1]}}(b)\right)_s \right\|_\beta \, \left\| x_{-s}
    \right\|_\alpha \\ \leq \, c_\alpha \, \left( 1 + \left| s \right|
    \right)^2 \, \left\| a \right\|_2 \, \left\| T_{f^{[1]}}(b)
    \right\|_\beta \, \\ \leq c_\alpha \, \left( 1 + \left| s \right|
    \right)^2 \, \left\| T_{f^{[1]}} \right\|_\beta \, \left\| a
    \right\|_2 \, \left\| b \right\|_\beta \\ \leq c_\alpha \, \left(
      1 + \left| s \right| \right)^2 \left\| T_{f^{[1]}}
    \right\|_\beta \, (1 + \epsilon).
  \end{multline*}
From~(\ref{Gfunction}), we derive $$ \left| S_1 \right| \leq \, c_\alpha \, (1
  + \epsilon)\, \left\| T_{f^{[1]}} \right\|_\beta \, \int_\Rl \left(
    1 + \left| s \right| \right)^2 \, \left| g(s) \right| \, ds \leq
  \, c_\alpha \, \left( 1 + \epsilon \right)\, \left\| T_{f^{[1]}}
  \right\|_\beta. $$ Observing that~$\epsilon > 0$ was arbitrary, we
  finally arrive at $$ \left| S_1 \right| \leq \, c_\alpha \, \left\|
    T_{f^{[1]}} \right\|_\beta. $$

The estimate of the term~$S_2$ is even simpler due to the fact that
the element~$a$ belongs to the Hilbert space~$L^2$ and, therefore, the
factor~$\left\| T_{f^{[1]}} \right\|_2$ does not exceed $\|f'\|_\infty\leq 1$ by Remark \ref{phiHS}. In other words, we have $$ \left| S_2 \right| \leq \,c_\alpha. $$

Combining the estimates for~$S_1$ and~$S_2$, we finally obtain $$
\left| \tau \left( y T_{f^{[1]}}(x)\right) \right| \leq \, c_\alpha
\, \left( 1 + \left\| T_{f^{[1]}} \right\|_\beta \right). $$
The lemma is completely proved.
\end{proof}

\begin{proof}[Proof of Theorem~\ref{PSLipschitzAlt}]
  Without loss of generality we assume that $\left\| f'
  \right\|_\infty \leq 1$.  Due to complex interpolation and duality,
  it is sufficient to give the proof when~$\alpha$ is ``close''
  to~$\infty$.

  Fix~$2 < \alpha < \infty$ sufficiently large and fix~$2 < \beta <
  \infty$ such that $\frac 12 = \frac 1\alpha + \frac 1\beta$. The
  exponent~$\beta$ approaches~$2$ as~$\alpha \rightarrow \infty$.
  Now, on one hand, Lemma~\ref{ImprovedDOI-base-rec} gives $$ \left\|
    T_{f^{[1]}} \right\|_\alpha \leq c_\alpha \left( 1 + \left\|
      T_{f^{[1]}} \right\|_\beta \right) $$ and, on the other hand,
  the complex interpolation method \cite{BL,DDP} provides
  \footnote{Note that $\theta$ approaches~$0$ as~$\alpha$ runs
    to~$\infty$; note also that~$\left\| T_{f^{[1]}} \right\|_2 \leq
    1$.} $$ \left\| T_{f^{[1]}} \right\|_\beta \leq \left\|
    T_{f^{[1]}} \right\|_\alpha^\theta,\ \ \text{where}\ \ \theta =
  \frac {2^{-1} - \beta^{-1}}{2^{-1} - \alpha^{-1}}. $$ The latter two
  inequalities, combined together, imply $$ \left\|T_{f^{[1]}}
  \right\|_\alpha \leq c_\alpha,$$ completing the proof.
\end{proof}

\subsection*{Reduction from higher to lower order polynomial integral
  momenta.}

{\coloremphasize The previous section explained our line of attack of
  the simplest polynomial integral momentum, the first divided
  difference~$f^{[1]}$.  We shall now say a few words explaining our
  approach to higher order polynomial integral momenta.  These will be
  reduced to the polynomial integral momenta of lower orders.  We
  shall demonstrate this using another simple example of (second
  order) polynomial integral momentum, the second order divided
  difference~$f^{[2]}$.  The reduction is based on the following
  representation for~$f^{[2]}$ (a similar but complex reduction
  formula of general polynomial integral momenta is given in
  Lemma~\ref{DecomLemmaII}):}
\begin{multline}
\label{SchemeReduction}
f^{[2]}(l, k, m) = \frac {l - k}{l - m} \,\phi_2 (l, k) + \frac {k -
  m}{l - m} \, \phi_2(m, k), \\ \text{where}\quad \phi_2 (l, m) =
\frac{f^{[1]}(l, m) - f'(m)}{l - m}.
\end{multline}
For simplicity assume that $l \leq k \leq m$ and $l \neq m$.  To see
the case~$n = 2$, we have to prove that the 
multilinear transformations given by Definition \ref{nmoi}
for both functions $$ \frac {l - k}{l - m} \, \phi_2 (l, k)
\ \ \text{and} \ \ \frac {k - m}{l - m} \, \phi_2 (m, k) $$ are
bounded.  Let us consider the first summand in the decomposition for
$f^{[2]}$ above.  Lemma~\ref{DecompositionLemma} gives $$ \frac {l -
  k}{l - m} = \int_{\Rl} g(s)\, \left( l - k \right)^{is} \left( l - m
\right)^{-is}\, ds,\ \ l < k \leq m, $$ which implies that we need to
study the operator $$ T_s (x, y) = \sum_{l < k \leq m} \left( k - l
\right)^{is} \left( m - l \right)^{-is} \phi_2 (l, k)\, E_l x E_k y
E_m. $$ If~$R_s$ is the mapping~$x \mapsto x_s$ from
Lemma~\ref{MarcCorl}, then, by Lemma~\ref{MOI-algebra}
(\ref{MOI-A-composition}), $$ T_s (x, y) = R_{-s} \left( T_{\phi_2}
  \left( R_s (x) \right) y \right). $$ Since~$R_s$ is bounded (see
Lemma~\ref{MarcCorl}), $T_s$ is also bounded, provided~$T_{\phi_2}$ is
bounded, i.e., we have reduced the question on the triple operator
integral down to the question on the double one.

We shall present a detailed technical account of the scheme above in
the following section.  We shall use the method of mathematical
induction.  The base of induction is proving that the operator
associated with~$\phi_2$, or in general, with~$\phi_m$, introduced
below, is bounded.  This part is done in Theorem~\ref{ImprovedDOI}.
It follows the lines of the proof of Theorem~\ref{PSLipschitzAlt},
with appropriate adjustments. The step of induction is reduction as
in~(\ref{SchemeReduction}).  This part is done in
Lemma~\ref{DecomLemmaII}. The scheme is finalised in the proof of
Theorem~\ref{MainTheoremExt}.

\section{Proof of Theorem \ref{MainTheoremCor}}
\label{sec:doi-fn}

In this section, we prove Theorem \ref{MainTheoremCor}, and this
requires some preparation.

\subsection*{Polynomial integral momenta.}
\label{sec:IPM}

Let $\Pl_n$ be the class of polynomials of~$n$ variables with real
coefficients.  Let~$\kappa > 0$ and let~$S^\kappa_n$ be the simplex $$
S^\kappa_n = \left\{ \left( s_0, \ldots, s_n \right) \in \Rl^{n + 1}:\
  \ \sum_{j = 0}^n s_j = \kappa,\ \ s_j \geq 0,\ \ 0 \leq j \leq n
\right\}. $$ We equip the simplex~$S^\kappa_n$ with the finite measure~$d\sigma_n$ defined by
\begin{equation}
\label{minv}
\int_{S^\kappa_n} \phi(s_0, \ldots, s_n)\, d\sigma_n =
\int_{R^\kappa_n} \phi\left( s_0, \ldots, s_{n-1}, \kappa - \sum_{j =
    0}^{n-1} s_j \right)\, dv_n,
\end{equation}
for every continuous function~$\phi: \Rl^{n+1} \mapsto \Cx$,
where
\begin{equation*}
  \label{Rn}
  R^\kappa_n = \left\{(s_0, \ldots, s_{n-1})
    \in \Rl^n:\ \ \sum_{j = 0}^{n-1} s_j \leq \kappa,\ \ s_j \geq 0,\ 0
    \leq j \leq n\right\}
\end{equation*}
and~$dv_n$ is the Lebesgue measure on~$\Rl^n$. The multiple integrals
in~(\ref{minv}) can be reduced to iterated integrals, which is
demonstrated in the proof of Lemma \ref{ddpim} below.  It can be seen
via a straightforward change of variables in~(\ref{minv}) that the
measure~$d\sigma_n$ is invariant under any permutation of the
variables~$s_0, \ldots, s_n$.  We set~$S_n:=S^1_n$ and~$R_n:=R^1_n$.

Let $$ \tilde s = \left( s_1, \ldots, s_n \right) \in R_n,\ \ (s_0,
\tilde s) \in S_n,\ \ s_0=1-\sum_{j=1}^n s_j. $$

Given $h \in C_b$, and $p \in \Pl_n$, we introduce
\begin{equation}
\label{2stars}
\phi_{n, h, p} \left( \tilde \lambda \right) = \int_{S_n} p\left(
  \tilde s \right)\, h \left( \sum_{j = 0}^n s_j \lambda_j \right)
\, d\sigma_n,
\end{equation}
where $\tilde \lambda = \left( \lambda_0, \ldots, \lambda_n \right)
\in \Rl^{n+1}$.  We shall call the function~$\phi_{n, h, p}$ a {\it
  polynomial integral momentum}. The function $\phi_{n, h, p}$ is
continuous since $h$ is.

The following routine fact shows that the polynomial integral momentum
is a generalization of the divided difference.

\begin{lemma}\label{ddpim}
For $f\in C^n$, $f^{[n]} = \phi_{n, f^{(n)}, 1}$.
\end{lemma}

\begin{proof}
We have that
\begin{align*}
\phi_{n, f^{(n)}, 1}&=\int_{S_n} f^{(n)} \left( \sum_{j = 0}^n s_j \lambda_j
    \right)\,d\sigma_n\\&=\int_{S_n} f^{(n)} \bigg(\sum_{j=0}^n
      s_j\lambda_n+\sum_{j=0}^{n-1} s_j(\lambda_{n-1}-\lambda_n)
      +\dots+\\
&\quad\quad\quad\quad\quad+\sum_{j=0}^1 s_j(\lambda_1 - \lambda_2) + s_0(\lambda_0 -\lambda_1)\bigg)\,d\sigma_n.
\end{align*}
  By substituting $1=\sum_{j=0}^n s_j$, $t_1=\sum_{j=0}^{n-1}s_j$,
  \dots, $t_{n-1}=\sum_{j=0}^1 s_j$, and $t_n=s_0$ in the latter
  integral, we obtain \[ \int_0^1 \, dt_1 \int_0^{t_1} \, dt_2 \dots
  \int_0^{t_{n-1}}f^{(n)}(\lambda_n+(\lambda_{n-1}-\lambda_n)t_1
  +\dots+(\lambda_0-\lambda_1)t_n)\,dt_n,\] which equals
  $f^{[n]}(\lambda_0,\lambda_1,\ldots,\lambda_n)$ by
  \cite[Formula~(7.12)]{dVRLG1993}.
\end{proof}

\begin{lemma}
  \label{PhiMHprop}
  If\/~$h \in \W_0$ and $p\in P_n$, then~$\phi_{n, h, p} \in \cC_n$ and $$
  \left\| \phi_{n, h, p} \right\|_{\cC_n} \leq c_p\, \left\| h
  \right\|_{\W_0}. $$
\end{lemma}

\begin{proof}[Proof of Lemma~\ref{PhiMHprop}]
Since $h\in W_0$, we have $g := \hat h \in L^1(\Rl)$, i.e., $$ h (t) = \int_\Rl
  g(s) e^{ist}\, ds,\ \ g \in L^1(\Rl).$$
  Observing that $$ h\left( \sum_{j = 0}^n s_j \lambda_j
  \right) = \int_{\Rl} \prod_{j = 0}^n e^{i s s_j \lambda_j}\, g(s)\,
  ds , $$ implies $$ \phi_{n, h, p} \left( \tilde \lambda \right) =
  \int_{\Rl} \int_{S_n} \prod_{j = 0}^n e^{is s_j \lambda_j}\,
  p(\tilde s)\, g(s) \, d \sigma_n\, ds. $$ Thus, choosing the finite
  measure space~$\left( \Rl \times S_n, g(s)\, ds \times p(\tilde s)\,
    d\sigma_n \right)$ and functions $$ a_j (s, \tilde s, t) = e^{i s
    s_j t},\ \ 0 \leq j \leq n, $$ in~(\ref{Crep}), we obtain
  that~$\phi_{n, h, p} \in \cC_n$ and $$ \left\| \phi_{n, h, p}
  \right\|_{\cC_n} \leq \frac{1}{n!} \, \left\| h \right\|_{\W_0}\,
  \sup_{\tilde s \in R_n} \left| p\left( \tilde s \right) \right|,$$ completing the proof.
\end{proof}

It follows immediately from Lemmas~\ref{ClassNewDef} and~\ref{PhiMHprop} that the
multilinear transformation~$T_{\phi_{n, h, p}}$ associated with the
function~$\phi_{n, h, p}$ is bounded on every~$L^{\alpha_1} \times
\ldots \times L^{\alpha_n}$, with $0 \leq \frac 1{\alpha_1} + \ldots + \frac 1{\alpha_n} \leq 1$, provided~$h \in \W_0$. The principal step toward Theorem \ref{MainTheoremCor} is the following improvement of the observation above.

\begin{theorem}
  \label{MainTheoremExt}
  Let~$p \in \Pl_n$ and $h \in C_b$. Let $1 < \alpha_j < \infty$, for
  $1 \leq j \leq n$, be such that $ 0 < \frac 1 {\alpha_1} + \ldots +
  \frac 1 {\alpha_n} < 1$. Then the operator~$T_{\phi_{n, h, p}}$
  exists and is bounded on every~$L^{\alpha_1} \times \ldots \times
  L^{\alpha_n}$ and $$ \left\| T_{\phi_{n, h, p}} \right\| \leq \,
  c_p\, \left\| h \right\|_\infty, $$ where the constant~$c_p > 0$ depends only on the exponents~$\alpha_j$, $j = 1, 2,
  \ldots, n$ and the polynomial~$p$.
\end{theorem}

\begin{remark}\label{MTER}
  Theorem~\ref{MainTheoremExt} and Lemma \ref{ddpim} imply
  that, for every~$f \in C^n$, $T_{f^{[n]}}$ is bounded
  on~$L^{\alpha_1} \times \ldots \times L^{\alpha_n}$, with~$0 < \frac
  1{\alpha_1} + \ldots + \frac 1{\alpha_n} < 1$ and $$
  \left\|T_{f^{[n]}}\right\|\leq \,c_{n}\,\left\| f^{(n)}
  \right\|_\infty. $$
\end{remark}

\subsection*{Theorem~\ref{MainTheoremExt} and discrete spectral measures.}

Before proceeding to the principal part of the proof of
Theorem~\ref{MainTheoremExt}, let us observe that essentially it needs
only to be proved in the case when the spectrum of the operator~$H$ is
concentrated at integral points. As soon as that is done, the rest of
the proof is a straightforward approximation. This observation is
formalized in the following lemma.

\begin{lemma}
  \label{ToDiscreteMeasure}
  Let~$n \in \N$, $p \in \Pl_n$, and let~$\left\{E_l \right\}_{l\in
    \Z}$ be a sequence of pairwise orthogonal spectral projections
  such that~$\sum_{l \in \Z} E_l = 1$.  Let~$1 < \alpha,\alpha_j <
  \infty$, $x \in L^{\alpha_j}$, for $1 \leq j \leq n$, and $0 <
  \frac1\alpha=\frac 1{\alpha_1} + \ldots + \frac 1{\alpha_n} <
  1$. If, for every~$h \in C_b$, the series
  \begin{equation}
    \label{ToDiscreteMeasureSeriesI}
    S_{h} (\tilde x): = \sum_{l_0, \ldots, l_n \in \Z} \phi_{n, h, p}
    \left(l_0, l_1, \ldots, l_n \right)\, E_{l_0} x_1 E_{l_1} x_2
    \cdot\ldots \cdot x_n E_{l_n},
  \end{equation}
  where~$\tilde x = \left( x_1, \ldots, x_n \right)$, converges
  in~$L^\alpha$, the mapping $$ \tilde x \mapsto S_h (\tilde x) $$ is
  bounded on $L^{\alpha_1} \times \ldots\times L^{\alpha_n}$
  independently of the sequence $\left\{E_l \right\}_{l\in \Z}$ with
  the norm $$\left\| S_h \right\| \leq c_p\, \left\| h
  \right\|_\infty, $$ then the claim of Theorem~\ref{MainTheoremExt}
  holds.
\end{lemma}

\begin{proof}[Proof of Lemma~\ref{ToDiscreteMeasure}]
  We observe that, by Lemma~\ref{CompactSupportTphi}, it is enough to
  consider $\phi_{n,h,p}$ only on compact subsets of $\Rl^{n+1}$ and,
  therefore the sum in~(\ref{ToDiscreteMeasureSeriesI}) is finite.

  Let~$h \in C_c$. Fix~$n \in \N$, $p \in \Pl_n$.  Observe that the
  assumption of the lemma implies that the series
  \begin{equation}
    \label{ToDiscreteMeasureSeriesII}
    y_m : = \sum_{l_0,\ldots, l_n \in \Z} \phi_{n, h, p} \left( \frac
      {l_0}m, \ldots,
      \frac {l_n}m \right)\, E_{l_0, m} x_1 E_{l_1, m} x_2 \cdot
    \ldots\cdot x_n E_{l_n, m}
  \end{equation}
  converges in~$L^\alpha$ and the mappings $$ S_{h, m} : \tilde x
  \mapsto y_m $$ from Definition~\ref{nmoi} are uniformly bounded
  multilinear operators on $L^{\alpha_1}\times\ldots\times
  L^{\alpha_n}$, with
  \begin{equation}
    \label{ToDiscreteMeasureSmEst}
    \left\| S_{h, m} \right\| \leq \, c_p \,\left\| h \right\|_\infty.
  \end{equation}
  Indeed, in order to see that~(\ref{ToDiscreteMeasureSeriesII}) fits
  to~(\ref{ToDiscreteMeasureSeriesI}) we just need to take
  $E_l=E_{l,m}$, $l\in\Z$, and replace $h$
  in~(\ref{ToDiscreteMeasureSeriesI}) with the function
  $h_1(t)=h\left( \frac t m \right)$, $t\in\Rl$.

  Thus, we see that to finish the proof of the lemma, we have only to
  show that for every fixed collection~$\{x_j\}_{j=0}^n \subset
  L^{\alpha_j}$, the sequence~$\left\{y_m \right\}_{m \geq 1}$
  converges in~$L^\alpha$.

  We shall show that the sequence~$\left\{ y_m \right\}_{m \geq 1}$ is
  Cauchy.  Fix~$\epsilon > 0$ and fix a $C_c^\infty$-function~$\tilde
  h$ such that
  \begin{equation}
    \label{ToDiscreteApprox}
    \left\| h - \tilde h\right\|_\infty \leq \frac
    \epsilon{c_p},
  \end{equation}
  where~$c_p$ is taken from~(\ref{ToDiscreteMeasureSmEst}).
  Let~$\tilde y_m := S_{\tilde h, m} (\tilde x)$.  By
  Lemma~\ref{PhiMHprop}, $\phi_{n,\tilde h, p}\in\cC_n$, and, hence,
  Definition \ref{nmoi} and Lemma~\ref{ClassNewDef} imply that the
  sequence~$\left\{\tilde y_m\right\}_{m \geq 1}$ is Cauchy.  That is,
  there is~$m_\epsilon \in \N$ such that
  \begin{equation}
    \label{ToDiscreteMeasurePreCauchy}
    \left\|\tilde y_m-\tilde y_{m'}\right\|_\alpha \leq \epsilon,\quad
    \text{for } m, m'
    > m_\epsilon.
  \end{equation}
  Since $$ y_m - \tilde y_m = S_{h - \tilde h, m} (\tilde x), $$
  by~(\ref{ToDiscreteMeasureSmEst}) and~(\ref{ToDiscreteApprox}) $$
  \left\| y_m - \tilde y_m \right\|_\alpha \leq \epsilon,\quad m \geq
  1. $$ Combining the latter with~(\ref{ToDiscreteMeasurePreCauchy})
  implies that for every~$\epsilon > 0$, there is~$m_\epsilon \in \N$
  such that for all $m, m' > m_\epsilon$,
  $$
  \left\| y_m - y_{m'} \right\|_\alpha \leq \left\| y_m - \tilde y_m
  \right\|_\alpha + \left\| \tilde y_m - \tilde y_{m'} \right\|_\alpha
  + \left\| \tilde y_{m'} - y_{m'} \right\|_\alpha \leq 3 \epsilon.
  $$
  The lemma is proved.
\end{proof}

Now we shall prove Theorem~\ref{MainTheoremExt} following the scheme
outlined in Section \ref{sec:doi-case}.

By Lemma~\ref{ToDiscreteMeasure}, without loss of generality, we
assume that~$H_t$ has pure integral point spectrum, i.e.,
if~$dE_\lambda$ is the spectral measure of~$H$, then $$ E(B) = \sum_{l
  \in B\cap\Z} E_l,\ \ \text{$B \subseteq \Rl$ is Borel,} $$ where~$E
= \left\{E_l \right\}_{l \in \Z}$ is a spectral family. The
operator~$T_{\phi_{n, h, p}}$ associated with $H$ is given by the
finite multiple sum
\begin{equation}
\label{f1}
T_{\phi_{n, h, p}} (x_1, \ldots,x_n) = \sum_{l_0, \ldots, l_n \in \Z}
\phi_{n, h, p} \left( l_0,
  \ldots, l_n \right)\, E_{l_0} x_1 E_{l_1} x_2 \cdot \ldots \cdot x_n
E_{l_n}.
\end{equation}
By Lemma \ref{CompactSupportTphi}, it is enough to consider the
polynomial integral momenta $\phi_{n,h,p}$ only on compact
sets. Therefore, $T_{\phi_{n,h,p}}$ is given by the finite sum in the
proofs below.

\subsection*{The base of induction.}
\label{sec:ImprovedDOI}

Note that upon integrating by parts, the function
$\phi_2(\lambda,\mu)$ in~(\ref{SchemeReduction}) can be expressed as
\begin{align*}
  \phi_2(\lambda,\mu)&=\frac{f^{[1]}(\lambda,\mu)-f'(\mu)}{\lambda -\mu}\\
  &=\frac{1}{\lambda-\mu}\left(tf'(\lambda+(\mu-\lambda)t)\big|_0^1-f'(\mu)\right)+\int_0^1
  tf''(\lambda+(\mu-\lambda)t)\,dt\\
  &=\int_0^1 tf''(\lambda+(\mu-\lambda)t)\,dt.
  \end{align*}
The function $\phi_2(\lambda,\mu)$ is a particular case of the more general function
\begin{equation}
\label{star}
  \phi_{m,h}(\lambda, \mu) = \int_0^1 t^{m-1} h(\lambda + (\mu -
  \lambda) t)\, dt,
\end{equation}
where $h \in C_b$ and $m\in\N$. Note that $\phi_{m,h}$, in its turn,
is the special case of~$\phi_{n, h, p}$ given by~(\ref{2stars})
with~$n = 1$ and~$p(t) = t^{m-1}$.  In particular, if~$m = 1$,
then~$\phi_{1,f'} = f^{[1]}$ (see, e.g., Lemma \ref{ddpim}).

The following theorem strengthens Theorem~\ref{PSLipschitzAlt}.

\begin{theorem}
  \label{ImprovedDOI}
  Let $h \in C_b$, $m\in \N$, and let $\phi_{m,h}$ be as in~(\ref{star}). Then the operator~$T_{\phi_{m,h}}$ is bounded on every~$L^\alpha$, $1<\alpha<\infty$, and
  $$ \left\| T_{\phi_{m,h}} \right\|_\alpha \leq \,c_{\alpha, m} \, \left\| h \right\|_\infty.$$
\end{theorem}

The special case~$m = 1$ of Theorem~\ref{ImprovedDOI} is equivalent to Theorem~\ref{PSLipschitzAlt} with $h=f'$. For $m>1$, the proof of Theorem~\ref{ImprovedDOI} repeats the one of Theorem~\ref{PSLipschitzAlt}, with technical modifications. Recall that the key step in the proof of Theorem~\ref{PSLipschitzAlt} is the decomposition~(\ref{PhiMrepTriv}).  For Theorem~\ref{ImprovedDOI}, we shall need the following extension of~(\ref{PhiMrepTriv}).

\begin{lemma}
  \label{PhiMrep}
  Let $h\in C_b$, $m\in \N$, and let $\phi_{m,h}$ be as in~(\ref{star}).
 Suppose that $\lambda \leq \xi \leq \mu$ and $\lambda \neq \mu$. Then
  \begin{align}
    \label{PhiMrepObj}
    \nonumber
    \phi_{m,h}(\lambda,\mu)&=\left(\frac{\lambda-\xi}{\lambda-\mu}\right)^m \phi_{m,h}(\lambda, \xi) + \left(\frac{\xi-\mu}{\lambda-\mu}\right)^m \phi_{m,h}(\xi, \mu)\\
    &\quad+\sum_{k = 1}^{m-1} C_{m-1}^{k-1}\left(\frac{\lambda-\xi}{\lambda-\mu}\right)^{m-k}
    \left(\frac{\xi-\mu}{\lambda-\mu}\right)^k \phi_{k,h}(\xi,\mu).
  \end{align}
\end{lemma}

\begin{proof}[Proof of Lemma~\ref{PhiMrep}]
Denote
$\zeta=\frac{\lambda-\xi}{\lambda-\mu}$ and $\omega=\frac{\xi-\mu}{\lambda-\mu}$.
We start with splitting
  \begin{equation}
  \label{PhiMrepObj1}
  \phi_{m,h}(\lambda,\mu)=\int_0^1 t^{m-1}h(\lambda+(\mu-\lambda)t)\,dt=\int_0^\zeta+\int_\zeta^1.
  \end{equation}
  We compute these integrals separately. Substituting $t = \zeta t_1$ in the first integral gives
  \begin{equation}
  \label{PhiMrepObj2}
  \int_0^\zeta = \int_0^1 \zeta^{m-1} t_1^{m-1} h (\lambda + (\xi -
  \lambda) t_1)\, \zeta d t_1 = \zeta^m \phi_{m,h} (\lambda, \xi).
  \end{equation}
Substituting~$t = \zeta + \omega t_2$ in the second integral and using
the Newton binomial formula gives
  \begin{multline}
    \int_\zeta^1 = \int_0^1 (\zeta + \omega t_2)^{m-1} h (\xi + (\mu -
    \xi) t_2)\, \omega dt_2 \\ =  \int_0^1 \omega\left[ \sum_{k =
        0}^{m-1} C_{m-1}^k \zeta^{m-k -1} \omega^{k} t_2^k \right]\,
    h(\xi + (\mu - \xi) t_2)\, dt_2 \\ = \int_0^1 \left[ \omega^m
      t_2^{m-1} + \sum_{k = 1}^{m-1} C_{m-1}^{k-1} \zeta^{m - k}
      \omega^k t_2^{k-1} \right]\, h(\xi + (\mu - \xi) t_2)\, dt_2 \\
   \label{PhiMrepObj3}
    = \omega^m \phi_{m,h} (\xi, \mu) + \sum_{k = 1}^{m-1} C_{m-1}^{k-1}
    \zeta^{m-k} \omega^k \phi_{k,h} (\xi, \mu).
  \end{multline}
Combining~(\ref{PhiMrepObj1}) - (\ref{PhiMrepObj3}) gives~(\ref{PhiMrepObj}).
\end{proof}

Now we prove Theorem~\ref{ImprovedDOI}.

\begin{proof}[Proof of Theorem~\ref{ImprovedDOI}]
  It is sufficient to prove the theorem for real-valued~$h \in C_b$
  such that~$\left\| h \right\|_\infty \leq 1$.

  Our objective is to show existence of a constant~$c_{\alpha, m} > 0$
  such that
  \begin{equation}
    \label{ImprovedDOIObj}
    \left| \tau \left( y T_{\phi_{m,h}}(x)\right) \right|
    \leq \, c_{\alpha, m} \, \left( 1 + \left\|
        T_{\phi_{m,h}}\right\|_\beta \right)
  \end{equation}
  for $x \in L^\alpha$ and $y \in L^{\alpha'}$ with~$\left\| x
  \right\|_\alpha = 1$ and~$\left\| y \right\|_{\alpha'} = 1$, where
  $\frac 12 = \frac 1\alpha + \frac 1 \beta$. As in the case of
  Lemma~\ref{ImprovedDOI-base-rec}, it is enough to prove
  (\ref{ImprovedDOIObj}) for $x$ being off-diagonal upper-triangular
  and $y$ being lower-triangular with $\tau$-finite left and right
  supports.  If~(\ref{ImprovedDOIObj}) is proved, then the remaining
  argument is similar to the verbatim repetition of the extrapolation
  trick in the proof of Theorem~\ref{PSLipschitzAlt}. The proof
  of~(\ref{ImprovedDOIObj}) is (only) computationally more difficult
  argument than the one in Lemma~\ref{ImprovedDOI-base-rec}, where the
  relation~(\ref{PhiMrepTriv}) is replaced with~(\ref{PhiMrepObj}).

Fix~$\epsilon > 0$ and  factorize~$y = a b$, where $a \in L^2$ and $b \in L^\beta$ are
  lower-triangular such that
  \begin{equation*}
    \label{PLfactorization1}
    1 \leq \left\| a \right\|_2 \, \left\| b \right\|_\beta \leq
    1 + \epsilon.
  \end{equation*}
  Keeping the same notation~$z_{\lambda\mu} = E_\lambda z E_\mu$, for
  every~$z \in M$, and arguing exactly as in
  Lemma~\ref{ImprovedDOI-base-rec}, we have $$ \tau \left( y
    T_{\phi_{m,h}} (x)\right) = \tau \left( a b
    T_{\phi_{m,h}}(x)\right) = \sum_{\lambda \leq \xi \leq \mu \atop
    \lambda \neq \mu} \phi_{m,h}(\lambda, \mu) \, \tau \left(
    a_{\mu\xi} b_{\xi\lambda} x_{\lambda\mu} \right). $$ Let~$g$ be
  the function from Lemma~\ref{DecompositionLemma}. Denote
  $x_s:=\sum_{\lambda<\mu}(\mu-\lambda)^{is}x_{\lambda\mu}$ and
  $b_s:=\sum_{\lambda<\xi}(\xi-\lambda)^{is}b_{\xi\lambda}$.  Applying
  Lemma \ref{MOI-algebra} and the representation~(\ref{PhiMrepObj})
  yields
  \begin{multline}
    \label{ImprovedDOIRep}
    \tau \left( y T_{\phi_{m,h}}(x)\right) = \int_\Rl g(s)\, \Biggl[
    \tau \left( a \left(T_{\bar \phi_{m,h}}(b)\right)_{ms} x_{-ms}
    \right) + \tau \left( \left(T_{\bar \phi_{m,h}}(a)\right)_{ms}b
      x_{-ms} \right) \\ + \sum_{k = 1}^{m-1} C_{m-1}^{k-1} \tau
    \left(T_{\phi_{k,h}} \left( a_{ks}\right) b_{ms - ks} x_{-ms}
    \right)\Biggr]\, ds,
  \end{multline}
  where $\bar \phi_{m,h}(\lambda,\mu):=\phi_{m,h}(\mu,\lambda)$.  First,
  we estimate the integrand components. By employing Remark
  \ref{phiHS} and the representation~(\ref{PhiMrepObj1}) and recalling
  $1\leq k\leq m$,
\begin{align*}
  \|T_{\phi_{k,h}}(a_{ks})\|_2 \leq
  \|\phi_{k,h}\|_\infty\|a_{ks}\|_2\leq\|h\|_\infty\|a\|_2\leq
  \|a\|_2,
\end{align*}
we obtain
$$\left| \tau \left( T_{\phi_{k,h}} \left( a_{ks}\right) b_{ms-ks}
    x_{-ms} \right) \right| \leq c_{\alpha}\|a\|_2 \|b_{ms-ks}\|_\beta
\|x_{-ms}\|_\alpha \leq c_{\alpha} (1+|ms|)^2(1+\epsilon).$$ Note
that, by Lemma~\ref{MOI-algebra}~(\ref{MOI-A-involution}), we have
$\left\| T_{\phi_{m,h}}\right\| = \left\| T_{\bar
    \phi_{m,h}}\right\|$.  Arguing as in
Lemma~\ref{ImprovedDOI-base-rec} implies
  \begin{align*}
    \left| \tau \left( a \left(T_{\bar \phi_{m,h}}(b)\right)_{ms}
        x_{-ms} \right) \right| &\leq \left\| a \right\|_2 \, \left\|
      \left(T_{\bar \phi_{m,h}}(b)\right)_{ms}
    \right\|_\beta \left\| x_{-ms} \right\|_\alpha \\
    &\leq \, c_\alpha \left( 1 + \left| ms \right| \right)^2\, \left\|
      a \right\|_2 \left\| T_{\bar \phi_{m,h}}(b) \right\|_\beta
    \left\| x \right\|_\alpha  \\
    &\leq \, c_\alpha \left( 1 + \left| ms \right| \right)^2\, \left\|
      T_{\phi_{m,h}}
    \right\|_\beta\, \left\| a \right\|_2 \, \left\| b \right\|_\beta\\
    &\leq \, c_{\alpha} \left( 1 + \left| ms \right| \right)^2\,
    \left\| T_{\phi_{m,h}} \right\|_\beta (1+\epsilon).
  \end{align*}
  By letting~$\epsilon \rightarrow 0$, we arrive at $$ \left| \tau
    \left( a \left(T_{\bar \phi_{m,h}}(b)\right)_{ms} x_{-ms} \right)
  \right| \leq c_{\alpha} \left( 1 + \left| ms \right| \right)^2\,
  \left\| T_{\phi_{m,h}} \right\|_\beta. $$ Similarly, $$ \left| \tau
    \left( \left(T_{\bar \phi_{m,h}}(a)\right)_{ms} b x_{-ms} \right)
  \right| \leq c_{\alpha} \, \left( 1 + \left| ms \right|
  \right)^2.  $$ Employing the triangle inequality
  in~(\ref{ImprovedDOIRep}), we see that
  \begin{align*}
    \left| \tau \left( a b T_{\phi_{m,h}}(x)\right) \right| &\leq \,
    c_{\alpha} \left( 1 + \left\| T_{\phi_{m,h}}\right\|_\beta
    \right)\, \int_\Rl \left| g(s) \right|\, \left( 1 +
      \left| ms \right| \right)^2\, ds \\
    &\leq\, c_{\alpha, m} \, \left( 1 + \left\| T_{\phi_{m,h}}
      \right\|_\beta \right).
  \end{align*}
  This proves~(\ref{ImprovedDOIObj}) and, hence, completes the proof
  of the theorem.
\end{proof}

\subsection*{The induction step.}

We need an extension of the decomposition~(\ref{SchemeReduction}) to
the case of higher dimensions.  Let $$ \tilde s =\left( s_1,
  \ldots, s_n \right) \in R_n,\ \ (s_0, \tilde s) \in S_n,\quad \text{
  with } s_0=\kappa-\sum_{j=1}^n s_j.$$ Given $h \in C_b$, and $p\in
\Pl_{n + 1}$, we introduce
\begin{equation}
\label{3stars}
\psi_{n, h, p} (\zeta, \tilde \mu) = \int_{S_n} p(\zeta, \tilde s)\, h
\left(\sum_{j = 0}^n s_j \mu_j \right)\, d\sigma_n,
\end{equation}
where $\zeta \in \Rl$, $\tilde \mu = \left( \mu_0, \ldots,
  \mu_{n} \right) \in \Rl^{n +1}$.

\begin{lemma}
  \label{DecomLemmaII}
  Let~$n \geq 2$, $h \in C_b$ and let~$p \in \Pl_n$. Denote $\tilde \lambda = \left( \lambda_3,
      \ldots, \lambda_n \right) \in \Rl^{n - 2}$ and assume that $\lambda_0 \leq
  \lambda_2 \leq \lambda_1$, with $\lambda_0 \neq \lambda_1$. Then there are polynomials~$q, r
  \in \Pl_n$ depending only on~$p$ such that the function $\phi_{n,h,p}$ given by~(\ref{2stars}) equals
  \begin{align}
    \label{DecomLemmaIIObj}
\nonumber
&\phi_{n,h,p}\left(\lambda_0,\lambda_1,\lambda_2,\tilde
\lambda \right)\\&\quad=\psi_{n-1, h, q} \left(\frac{\lambda_0-\lambda_2}{\lambda_0-\lambda_1}, \lambda_0,\lambda_2, \tilde \lambda \right)+ \psi_{n-1, h, r} \left(
\frac{\lambda_0-\lambda_2}{\lambda_0-\lambda_1}, \lambda_1,\lambda_2, \tilde \lambda \right).
  \end{align}
\end{lemma}

We prove Lemma~\ref{DecomLemmaII} by reducing it to the special case $n=2$ discussed in the lemma below.

\begin{lemma}
  \label{IntegralRel}
  Let~$h \in C_b$ and let $m, k \in \N\cup\{0\}$. Let~$\kappa > 0$
  and $\lambda \leq \xi \leq \mu$, with $\lambda \neq \mu$. Then there are~$q, r \in \Pl_3$ such that
  \begin{multline*}
    \int_0^\kappa t^m dt \int_0^t s^k h(\kappa \xi + (\lambda - \xi) t + (
    \mu - \lambda) s)\, ds \\ = \int_0^\kappa q\left(\frac{\lambda-\xi}{\lambda -\mu}, \kappa, \theta\right)\, h(
    \kappa \xi + (\lambda - \xi) \theta )\, d\theta \\ + \int_0^\kappa
    r\left(\frac{\lambda-\xi}{\lambda -\mu}, \kappa, \sigma\right)\, h(\kappa \xi + (\mu - \xi) \sigma)\, d\sigma.
  \end{multline*}
  Here the polynomials~$q$ and~$r$ depend on~$m$ and~$k$, but do not depend on~$h$.
\end{lemma}

First, we prove the decomposition~(\ref{DecomLemmaIIObj}) and, then, auxiliary Lemma \ref{IntegralRel}.

\begin{proof}[Proof of Lemma~\ref{DecomLemmaII}]
 By the definitions of the functions $\phi_{n,h,p}$ and $\psi_{n-1, h, q}$, we have integration
over the simplex $S_n$ (for which $\sum_{j=0}^ns_j=1$) on both sides of~(\ref{DecomLemmaIIObj}).
Since $\sum_{j=0}^2s_j=1-\sum_{j=3}^ns_j$, we can split the integral over $S_n$ (which is defined on p. \pageref{Rn}) into the repeated integral
\begin{equation}
 \label{DecomLemmaIIFixInt}
 \int_{R_{n - 2}} ds_3 \ldots ds_n\int_{S_2^\kappa}ds_2ds_1ds_0,
 \end{equation}
where $\kappa=1-\sum_{j=3}^ns_j$. We fix the point~$\tilde s=\left( s_3, \ldots, s_n \right) \in R_{n - 2}$ (and, hence, fix $\kappa$ too) and consider the integrands over $S_2^\kappa$. On the left hand side, we have the integrand
  \begin{equation}
    \label{DecomLemmaIIIndI}
    \int_{S^\kappa_2} p_1 (s_0, s_1, s_2)
    h_1 \left( s_0 \lambda_0 + s_1 \lambda_1 + s_2 \lambda_2 \right)\,
    d \sigma_2,
  \end{equation}
  where we set $$p_1 (s_0, s_1,s_2) := p(s_1, s_2, s_2, \tilde s),\ \ h_1 (t) := h\left( t + \sum_{j
      = 3}^n s_j \lambda_j \right). $$ Note that
  \[s_0\lambda_0+s_1\lambda_1+s_2\lambda_2= (s_0+s_1+s_2)\lambda_2 +
  (s_0+s_1) (\lambda_0 - \lambda_2) + s_1 (\lambda_1-\lambda_0).\]
  Further, by recalling $\kappa=s_0+s_1+s_2$, making substitution
  $t=s_0+s_1$, and $s=s_1$ and setting~
  \begin{align}\label{la0la}
  \lambda_0=\lambda,\quad \lambda_1 = \mu,\quad \lambda_2 = \xi,
  \end{align}
we see that~(\ref{DecomLemmaIIIndI}) becomes
  \begin{equation}
    \label{DecomLemmaIIIndII}
    \int_0^\kappa dt \int_0^t p_2
    (t, s)\, h_1 (\kappa \xi + (\lambda - \xi) t + (\mu - \lambda) s)\,
    ds,
  \end{equation}
  where $$ p_2 (t, s) = p_1 (t - s, s, \kappa - t). $$ Note that~$p_2$
  is a polynomial of two variables~$t$ and~$s$, that is, it is a finite
  linear combination of monomials~$t^m s^k$ multiplied by certain
  powers of fixed $\tilde s$ (these powers are determined by
  $p$). Consequently, applying Lemma~\ref{IntegralRel}, we see
  that~(\ref{DecomLemmaIIIndII}) equals
  \begin{equation}
    \label{DecomLemmaIIIndIII}
    \int_0^\kappa q_1(\zeta, \kappa, \theta) h_1 (\kappa \xi + (\lambda - \xi)
    \theta)\, d\theta + \int_0^\kappa r_1(\zeta, \kappa, \sigma) h_1 (\kappa \xi +
    (\mu - \xi) \sigma)\, d\sigma,
  \end{equation}
where $\zeta=\frac{\lambda-\xi}{\lambda-\mu}$ and $q_1, r_1 \in \Pl_3$ are appropriate linear combinations of the polynomials from Lemma~\ref{IntegralRel} (the precise
representation for which is given in~(\ref{IntegralRelTemp1}) below).  If we
substitute now~$s_0 = \theta$, $s_2 = \kappa - \theta$ in the first
integral and~$s_1 = \sigma$, $s_2 = \kappa - \sigma$ in the second
one, and recall the equalities~(\ref{la0la}), then we see that~(\ref{DecomLemmaIIIndIII}) equals
\begin{align*}
&\int_{S_1^\kappa} q_1 \left(\frac{\lambda_0-\lambda_2}{\lambda_0-\lambda_1}, \kappa, s_0\right)\, h_1 \left( s_0 \lambda_0 +s_2 \lambda_2 \right) d\sigma_1\\
&\quad+ \int_{S_1^\kappa} r_1 \left(\frac{\lambda_0-\lambda_2}{\lambda_0-\lambda_1},
\kappa, s_1\right) h_1 \left( s_1 \lambda_1 + s_2 \lambda_2 \right)d\sigma_1.
\end{align*}
Recall that $p_2$ was obtained from $p$ by fixing the value of $\tilde s$. By
letting $\tilde s$ vary, we obtain from $q_1$ and $r_1$ in $P_3$ such
polynomials $q$ and $r$ in $P_n$ that
  \begin{align*}
    &\int_{S_2^\kappa}p(s_0,s_1,s_2,\tilde s)
    h\left(s_0\lambda_0+s_1\lambda_1+s_2\lambda_2+\sum_{j=3}^n s_j\lambda_j\right)\,d\sigma_2\\
    &\quad=\int_{S_1^\kappa}q\left(\frac{\lambda_0-\lambda_2}{\lambda_0-\lambda_1},\kappa,s_0,\tilde
    s\right)h\left(s_0\lambda_0+s_2\lambda_2+\sum_{j=3}^n s_j\lambda_j\right)\,d\sigma_1\\
    &\quad\quad+\int_{S_1^\kappa}r\left(\frac{\lambda_0-\lambda_2}{\lambda_0-\lambda_1},\kappa,s_1,\tilde
    s\right)h\left(s_1\lambda_1+s_2\lambda_2+\sum_{j=3}^n s_j\lambda_j\right)\,d\sigma_1.
  \end{align*}
Now we integrate the latter expression over $R_{n-2}$ with respect to $\tilde s$ and obtain
(\ref{DecomLemmaIIObj}).
\end{proof}

\begin{proof}[Poof of Lemma~\ref{IntegralRel}]
  We first prove the case~$\kappa = 1$.  Let us compute the integral on
  the left hand side
  \begin{equation}
  \label{IntegralRelTemp0}
  \text{LHS} := \int_0^1 t^m dt \int_0^t s^k
  h(\xi + (\lambda - \xi) t + ( \mu - \lambda) s)\, ds
  \end{equation}
by substituting $$ u = \xi + (\lambda - \xi) t + (\mu - \lambda) s. $$
  We have $$ \text{LHS} = \int_0^1 t^m dt \int_{\xi + (\lambda - \xi)
    t}^{\xi + (\mu - \xi) t} \left[ \frac { u - \xi - (\lambda - \xi)
      t}{\mu - \lambda} \right]^k\, \frac {h(u) \, du}{\mu -
    \lambda}. $$ Next, we observe that changing the order of integration
  yields $$ \int_0^1 dt \int_{\xi + (\lambda - \xi) t}^{\xi + (\mu -
    \xi) t} du = \int_\lambda^\xi du \int_{\frac { u - \xi}{\lambda -
      \xi}}^1 dt + \int_{\xi}^\mu du \int_{\frac {u - \xi}{\mu -
      \xi}}^1 dt. $$ With this change of the order, we obtain
  \begin{multline*}
    \text{LHS} = \int_{\lambda}^\xi \frac {h(u)\, du}{\mu -
      \lambda}\, \int_{\frac {u -\xi}{\lambda - \xi}}^1 \left[ \frac { u
        - \xi - (\lambda - \xi) t}{\mu - \lambda} \right]^k t^m\, dt
    \\ + \int_{\xi}^\mu \frac {h(u)\, du}{\mu - \lambda}\, \int_{\frac
      {u - \xi}{\mu - \xi}}^1 \left[ \frac { u - \xi - (\lambda - \xi)
        t}{\mu - \lambda} \right]^k t^m\, dt.
  \end{multline*}
  Further, we substitute $$\zeta=\frac{\lambda-\xi}{\lambda -\mu},\quad \theta = \frac {u - \xi}{\lambda - \xi},\quad\text{and}\quad \sigma = \frac {u - \xi}{\mu - \xi} $$ in the
  first and in the second integrals, respectively.  This gives
  \begin{multline}
    \label{IntegralRelTemp}
    \text{LHS} = \zeta \int_0^1 h(\xi + (\lambda - \xi) \theta)\, d\theta
    \int_{\theta}^1 \zeta^k \left( t - \theta \right)^k t^m \, dt \\ + (1
    - \zeta) \int_0^1 h(\xi + (\mu - \xi)\sigma)\, d\sigma \int_{\sigma}^1
    \left( (1 - \zeta) \sigma - \zeta t\right)^k t^m \, dt.
  \end{multline}
  Setting
  \begin{multline*}
    q(\zeta, \theta) := \zeta^{k + 1} \int_{\theta}^1 \left( t - \theta
    \right)^k t^m dt \quad\text{and}\\ r(\zeta, \sigma) := (1 - \zeta)
    \int_{\sigma}^1 \left( (1 - \zeta)\, \sigma - \zeta t\right)^k t^m \, dt
  \end{multline*}
  and observing that~$q, r \in \Pl_2$ finishes the proof when~$\kappa =1$.

  In order to see the case with arbitrary~$\kappa > 0$, we
  replace~$\lambda$, $\mu$, and~$\xi$ in~(\ref{IntegralRelTemp0})
  with~$\kappa \lambda_1$, $\kappa \mu_1$, and~$\kappa \xi_1$ and
  substitute $$ t_1 = \kappa t,\ \ s_1 = \kappa s,\ \ \theta_1 = \kappa
  \theta, \ \ \text{and}\ \ \sigma_1 = \kappa \sigma. $$ This gives
  \begin{multline*}
    \int_0^\kappa \frac {t_1^m dt_1}{\kappa^{m + 1}} \int_0^{t_1} h(\kappa
    \xi_1 + (\lambda_1 - \xi_1) t_1 + (\mu_1 - \lambda_1) s_1)\, \frac
    {s_1^k ds_1}{\kappa^{k + 1}} \\ = \zeta^{k + 1} \int_0^\kappa h(\kappa \xi_1
    + (\lambda_1 - \xi_1) \theta_1) \frac {d \theta_1}{\kappa}
    \int_{{\theta_1}}^\kappa (t_1 - \theta_1)^k t_1^m \frac
    {dt_1}{\kappa^{k + m + 1}} \\ + (1 - \zeta) \int_0^\kappa h(\kappa \xi_1 +
    (\mu_1 - \xi_1) \sigma_1) \, \frac {d\sigma_1}{\kappa}
    \int_{\sigma_1}^\kappa \left( (1 - \zeta) \sigma_1 - \zeta t_1 \right)^k
    t_1^m \frac {dt_1}{\kappa^{k + m + 1}}.
  \end{multline*}
  We cancel the factor~$\kappa^{-m-k-2}$ on both sides of the previous calculation and set
  \begin{multline}
  \label{IntegralRelTemp1}
    q(\zeta, \kappa, \theta_1) := \zeta^{k + 1} \int_{\theta_1}^\kappa (t_1 -
    \theta_1)^k t_1^m dt_1\quad\text{and}\\ r(\zeta, \kappa, \sigma_1) := (1
    - \zeta) \int_{\sigma_1}^\kappa \left( (1 - \zeta) \sigma_1 - \zeta t_1 \right)^k
    t_1^m dt_1.
  \end{multline}
  Observing $q, r \in \Pl_3$ finishes the proof in the general case $\kappa > 0$.
\end{proof}

We are now ready to complete the proof of Theorem~\ref{MainTheoremExt}.

\begin{proof}[Proof of Theorem~\ref{MainTheoremExt}]
  It is sufficient to prove the theorem for real-valued~$h \in C_b$.
  It is also sufficient to prove the theorem when~$H$ has spectrum at
  integral points (see Lemma \ref{ToDiscreteMeasure}) and consider
  $\phi_{n,h,p}$ only on compact sets (see Lemma
  \ref{CompactSupportTphi}).

  The proof is by the method of the mathematical induction with
  respect to~$n \in \N$.  The base of the induction, i.e., the case~$n
  = 1$, is done in Theorem~\ref{ImprovedDOI}.

  Let us assume that~$n > 1$ and that the theorem is proved for~$n -
  1$, that is, we have
  \begin{equation}
  \label{f2}
  \left\| T_{\phi_{n-1, h, p}} \right\| \leq c_p\, \left\| h
  \right\|_\infty.
  \end{equation}
  Let~$1 < \alpha_0 < \infty$ be such that $1 = \frac 1{\alpha_0} +
  \frac 1{\alpha_1} + \ldots + \frac 1{\alpha_n} $ and let~$x_j \in
  L^{\alpha_j}$ be such that~$\left\| x_j \right\|_{\alpha_j} \leq 1$,
  $0\leq j\leq n$.  We shall show that
  \begin{equation}
    \label{MainTheoremExtObj}
    \sup_{x_j \in L^{\alpha_j},\; \|x_j\|_{\alpha_j}\leq 1}\left| \tau
      \left( x_0 T_{\phi_{n, h, p}} \left(\tilde x \right) \right)
    \right| \leq c_p\, \left\| h \right\|_\infty.
  \end{equation}
  Recall that (see~(\ref{f1})) $$ \tau\left( x_0 T_{\phi_{n, h, p}}
    \left( \tilde x \right) \right) = \sum_{\tilde l \in \Z^{n + 1}}
  \phi_{n, h, p} \left( \tilde l \right) \tau \left( \prod_{j = 0}^n
    x_j E_{l_j} \right),\ \ \tilde l = \left( l_0, l_1, \ldots, l_n
  \right). $$

Let $D$ denote the diagonal subset of $\Z^{n + 1}$, that is,
\begin{equation*}
D=\{(l_0,l_1,\dots,l_n):\ l_0=l_1=\dots=l_n\}.
\end{equation*}
Since for $(l_0,l_1,\dots,l_n)\in D$,
\begin{equation*}
\phi_{n,h,p}(l_0,l_1,\dots,l_n)=h(l_0)\int_{S_n}p(\tilde s)\,d\sigma_n,
\end{equation*}
the 
multilinear transformation $T_{\phi_{n,h,p}}$ is diagonal and
trivially bounded by~$\left\| h \right\|_\infty$.  Let~$\epsilon =
(\epsilon_1, \epsilon_2, \ldots, \epsilon_n)$, $\epsilon_j = \pm 1$
and let~$K_\epsilon \subseteq \Z^{n+1}$ be such that
  \begin{equation*}
    K_\epsilon = \Bigl\{\left( l_0, \ldots, l_n \right)
    \in \Z^{n + 1}:\ \ l_{j-1} \leq l_j\
    \text{if~$\epsilon_j = 1$};\ \ l_{j-1} >
    l_j\ \text{if~$\epsilon_j = -1$},\ \ 1\leq j\leq n \Bigr\}.
  \end{equation*}

  The space~$\Z^{n + 1}\setminus D$ splits into the disjoint union of~$2^n$
  sets~$K_\epsilon$, $\epsilon\in\{-1,1\}^n$.
  Fix~$\epsilon \in \left\{-1, 1\right\}^n$.  There
  is an index~$0 \leq j_\epsilon \leq n$ such that
\begin{equation}
\label{jepsilon}
  \left( l_0,\ldots, l_n \right) \in K_\epsilon\ \ \Longrightarrow\ \
  l_{j_\epsilon - 1} \leq l_{j_\epsilon} \ \ \text{and}\ \  l_{j_\epsilon}> l_{j_\epsilon + 1},
\end{equation}
where the decrement and increment of the indices~$j_\epsilon - 1$ and~$j_\epsilon + 1$ are
  understood modulo~$n$, i.e., if~$j = 0$, then~$j -1 = n$ and if~$j =
  n$, then~$j + 1 = 0$. Next, by fixing~$j_\epsilon$, we further
  split~$K_\epsilon$ into subsets~$K_{\epsilon, i}$, $i = 0, 1$, where
  \begin{multline*}
    K_{\epsilon, 0} = \left\{ \left( l_0, \ldots,
        l_n\right) \in K_\epsilon:\ \ l_{j_\epsilon-1}
      \leq l_{j_\epsilon +
        1} \right\}\quad  \text{and} \\
    K_{\epsilon, 1} = \left\{ \left( l_0, \ldots,
        l_n\right) \in K_\epsilon:\ \ l_{j_\epsilon-1} >
      l_{j_\epsilon + 1} \right\}.
  \end{multline*}  The space~$\Z^{n + 1}\setminus D$
    splits into the disjoint union of~$2^{n + 1}$ sets~$K_{\epsilon,
      i}$, i.e., $$ \Z^{n + 1}\setminus D =
    \bigcup_{\epsilon\in\{-1,1\}^n}\,
  \bigcup_{i = 0, 1} K_{\epsilon,i}. $$  This means that
  $$T_{\phi_{n, h, p}} = \sum_{\epsilon} \sum_{i =
    0,1}T^{\epsilon,i}_{\phi_{n,h,p}},$$ where
  \begin{equation*}
T^{\epsilon, i}_{\phi_{n, h,p}}(\tilde x)=
    \sum_{\tilde l \in K_{\epsilon, i}} \phi_{n, h, p} \left( \tilde l \right)\,
    E_{l_0}x_1 E_{l_1}x_2\dots x_n E_{l_n}.
  \end{equation*}
We show~(\ref{MainTheoremExtObj}) for each~$T^{\epsilon,i}_{\phi_{n,h,p}}$.
This will finish the proof of the theorem.

We fix~$\epsilon = \left( \epsilon_1, \ldots, \epsilon_n
\right)\in\{-1,1\}^n$, fix the index $j_\epsilon\in\{0,1,\dots,n\}$ as in~(\ref{jepsilon}), and let
$i =0, 1$.  We then have
  \begin{equation*}
    \left( l_0, l_1, \ldots, l_n \right) \in
    K_{\epsilon, i}\ \ \Longrightarrow
    \begin{cases} l_{j_\epsilon-1} \leq
    l_{j_\epsilon + 1} < l_{j_\epsilon}& \text{ if } i = 0, \\
    l_{j_\epsilon + 1} < l_{j_\epsilon - 1}
    \leq l_{j_\epsilon}& \text{ if } i = 1.
    \end{cases}
  \end{equation*}
  By shifting and reversing if~$i = 1$ (as in Lemma~\ref{MOI-algebra}
  (\ref{MOI-A-involution}) and~(\ref{MOI-A-duality})) the enumeration
  of the variables~$l_j$ and operators~$x_j$, we may assume that $$
  l_0 \leq l_2 < l_1. $$

  Now we apply Lemma~\ref{DecomLemmaII}.  Let~$q, r \in P_{n}$ be such
  that $$ \phi_{n, h, p} (l_0, l_1, l_2, \ldots,l_n) = \psi_{n-1, h,
    r} (\zeta, l_1, l_2, \ldots,l_n) + \psi_{n-1, h, q} (\zeta, l_0,
  l_2, \ldots,l_n),$$ where $\zeta=\frac{l_0-l_2}{l_0-l_1}$.
  Let~$Q_{n-1, h,r}$ and~$R_{n-1, h, q}$ be the
 multilinear transformations corresponding to the right hand side functions~$\psi_{n-1,
    h, r}$ and~$\psi_{n-1, h, q}$ via Definition \ref{nmoi}.  From the above decomposition we
  obtain $$ T^{\epsilon, i}_{\phi_{n, h, p}} = Q_{n-1, h, r} + R_{n-1,
    h, q}. $$ We shall show that the
  estimate~(\ref{MainTheoremExtObj}) holds for both~$Q_{n-1, h, r}$
  and~$R_{n-1, h, q}$.  We show it for~$Q_{n-1, h, r}$; for~$R_{n-1,h,
    q}$ the proof is similar.

Let~$\tilde s = (s_1, \ldots, s_{n}) \in S_{n-1}$.  Since the
polynomial~$r(\zeta, \tilde s)$ is a linear combination of the
polynomials~$r_1(\zeta,\tilde s)=(1-\zeta)^m p_1 (\tilde s)$, with $m\geq 0$ and $p_1 \in \Pl_{n-1}$
(see the integration in~(\ref{IntegralRelTemp1})), it is sufficient to prove
(\ref{MainTheoremExtObj}) for~$Q_{n-1, h, r_1}$, i.e., where~$r$ in $Q_{n-1, h, r}$ is replaced with~$r_1$.

From~(\ref{2stars}) and~(\ref{3stars}), we have
$$\psi_{n-1, h, r_1}=(1-\zeta)^m \phi_{n-1, h, p_1}.$$
Let~$g$ be as in Lemma~\ref{DecompositionLemma} and let, as in Lemma~\ref{MarcCorl},
  \begin{equation*}
    x_{s, 1} = \sum_{l_0<l_1} (l_1-l_0)^{is} E_{l_0} x_1 E_{l_1}\ \
    \text{and}\ \  x_{s, 2} = \sum_{l_2<l_1}(l_1 - l_2)^{is} E_{l_2} x_2
    E_{l_1}.
  \end{equation*}
By Lemma~\ref{DecompositionLemma},
$$(1 - \zeta)^m = \left( \frac {l_2 - l_1} {l_0 -
    l_1} \right)^m = \int_\Rl g(s) \, \left( l_1 - l_2 \right)^{ms}
\left( l_1 - l_0 \right)^{-ms}\, ds.$$ Applying~(\ref{MOI-A-product})
and~(\ref{MOI-A-composition}) of Lemma~\ref{MOI-algebra} gives the
reduction $$Q_{n-1,h,r_1}\left(x_1,x_2,\tilde x\right)=\int_\Rl
g(s)\,x_{-ms,1}\,T_{\phi_{n-1,h,p_1}} \left(x_{ms,2},\tilde
  x\right)\,ds,\ \ \tilde x=(x_3,\ldots,x_n), $$ of the problem of
order $n$ to the problem of order $n-1$, i.e., $Q_{n-1,h,r_1}$ is the
integral of order $n$ and $T_{\phi_{n-1,h,p_1}}$ is the integral of
order $n-1$.

Recall that by the assumption of the induction, the
operator~$T_{\phi_{n-1, h, p_1}}$ is bounded (see~(\ref{f2})).  Recall
also that by Lemma~\ref{MarcCorl}, $$ \left\| x_{-ms,
    1}\right\|_{\alpha_1} \leq c_{\alpha_1} \, \left( 1 + \left| ms
  \right| \right) \ \ \text{and}\ \ \left\| x_{ms,
    2}\right\|_{\alpha_2} \leq c_{\alpha_2} \, \left( 1 + \left| ms
  \right| \right)^2. $$ Thus, $$ \left\| Q_{n-1, h, r_1} \right\| \leq
c_{p_1,\alpha_1,\dots,\alpha_n} \, \int_{\Rl} \left| g(s) \right|\,
\left( 1 + \left| ms \right| \right)^2\, ds\cdot\|h\|_\infty \leq
c_{m, p_1, \alpha_1,\dots,\alpha_n} \cdot\|h\|_\infty, $$ The latter
justifies~(\ref{MainTheoremExtObj}) for~$Q_{n-1, h, r_1}$ and, thus,
finishes the proof.
\end{proof}

We conclude our exposition with the proof of Theorem \ref{MainTheoremCor}.

\begin{proof}[Proof of Theorem \ref{MainTheoremCor}]
The facts $\frac{d^n}{dt^{n}}\left[f\left(H_t\right)\right]\in \mL^1$ and $\Delta_{n,f}(H,V)\in \mL^1$ can be derived from \cite{AzCaDoSu2009,PellerMult}.
We prove the second fact; the first one is even simpler.
 From~\cite[Theorem~1.43 and~1.45]{Schwartz1969}, we have
\begin{equation}
    \label{SSFRemainderTemp}
    \Delta_{n, f} (H, V) = \frac 1{(n-1)!}\, \int_0^1 (1 - t)^{n-1}
    \frac {d^n}{dt^{n}} \left[ f \left( H_t \right) \right]\, dt,
  \end{equation} where, by~\cite[Theorem~5.7]{AzCaDoSu2009}, the operator derivative is well defined and equals
\begin{equation}
    \label{SSFFrechetTemp}
    \frac {d^n}{dt^n} \left[ f(H_t) \right] = n!\, T_{t, f^{[n]}} \bigl(
    \underbrace{V, \ldots, V}_{\text{$n$-times}} \bigr).
\end{equation}
For the operator~$T_{t, f^{[n]}}$, by Lemmas \ref{ClassNewDef}, \ref{ddpim}, and \ref{PhiMHprop}, we have $\left\|T_{t, f^{[n]}}\right\|\leq \left\| f \right\|_{\W_n}$ and, therefore,
\begin{equation}
\label{norm1est}
\Delta_{n, f} (H, V) \in \mL^1  \ \ \text{and}\ \ \left\| \Delta_{n, f} (H, V) \right\|_{\mL_1} \leq \,
  \left\| f \right\|_{\W_n} \, \left\| V\right\|_n^n.
\end{equation}
We shall now justify the estimate~(\ref{TextPrincipalEstimate}).
 Using~(\ref{SSFFrechetTemp}) in~(\ref{SSFRemainderTemp}), we see that $$ \Delta_{n, f} (H, V) = n \, \int_0^1 (1 - t)^{n-1} T_{t,f^{[n]}} \bigl(\underbrace{V, \ldots,V}_{\text{$n$-times}}\bigr)\, dt. $$
By \cite[Lemma 3.10]{AzCaDoSu2009}, we have
\begin{equation}
\label{f3}
\tau\left(\Delta_{n,f}(H,V)\right)=
n\,\int_0^1(1-t)^{n-1}\tau\left(T_{t,f^{[n]}}\bigl(\underbrace{V,\ldots,
V}_{\text{$n$-times}}\bigr)\right)\,dt.
\end{equation}
Thus, to show~(\ref{TextPrincipalEstimate}), it is sufficient to see that there is a constant~$c_n > 0$ such that
  \begin{equation}
    \label{SSFRemainderObj}
    \left| \tau \left( T_{t,f^{[n]}} \bigl(\underbrace{V, \ldots,
          V}_{\text{$n$-times}}\bigr) \right) \right| \leq \, c_n \,
    \left\| f^{(n)}\right\|_\infty \, \left\| V\right\|_n^n,\ \ 0 \leq t\leq 1.
  \end{equation}
One can derive from the representation~(\ref{ClassNewDefRep}) for $T_{f^{[n]}}$ that
  \begin{equation}
    \label{SSFRemainderTempI}
    \tau \left( T_{t, f^{[n]}} \bigl(\underbrace{V, \ldots,
        V}_{\text{$n$-times}}\bigr) \right) = \tau \left( T_{t,
        \phi} \bigl(\underbrace{V,
        \ldots,V}_{\text{$(n-1)$-times}}\bigr)\, V \right),
  \end{equation}
  where~$T_{t, \phi}$ is the multiple operator integral associated
  with~$H_t$ and $$ \phi(\lambda_0, \lambda_1, \ldots, \lambda_{n-1})=
  f^{[n]} \left( \lambda_0, \lambda_0, \lambda_1, \ldots,\lambda_{n-1}
  \right)$$ (for more details, see, e.g., \cite[Lemma 3.8]{moissf}).
  Noting that~$\phi$ is~$\phi_{n-1, h, p}$ with~$h = f^{(n)}$
  and~$p(s_1, \ldots, s_{n-1}) = 1 - \sum_{j = 1}^{n-1} s_j$, from
  Theorem~\ref{MainTheoremExt}, we obtain
\begin{equation}
    \label{SSFRemainderKey}
    \left\| T_{t, \phi}\bigl(\underbrace{V, \ldots, V}_{\text{$(n-1)
          $-times}}\bigr)
    \right\|_{\frac n{n-1}} \leq c_n\, \left\| f^{(n)}\right\|_\infty\,
    \left\| V \right\|_n^{n-1},\ \ 0 \leq t \leq 1.
  \end{equation}
  Clearly, (\ref{f3}) and~(\ref{SSFRemainderKey}) combined
  with~(\ref{SSFRemainderTempI}) imply~(\ref{SSFRemainderObj}).
  Thus, (\ref{TextPrincipalEstimate}) follows.
To prove the continuity of $V\mapsto \tau\left(\Delta_{n,f}(H,V)\right)$, we notice that
\begin{align}
\label{f4}
\nonumber
&\left|\tau\left(\Delta_{n,f}(H,V_j) - \Delta_{n,f}(H,V_k)\right)\right|\\
&\quad=\left|\frac 1{(n-1)!}\, \int_0^1 (1 - t)^{n-1}
  \tau\left(T_{t,f^{[n]}}\bigl( \underbrace{V_j,
      \ldots,V_j}_{\text{$n$-times}} \bigr)
    -T_{t,f^{[n]}}\bigl(\underbrace{V_k,
      \ldots,V_k}_{\text{$n$-times}}\bigr)\right)\, dt\right|.
  \end{align}
  Similarly to~(\ref{SSFRemainderObj}), it follows from
  Theorem~\ref{MainTheoremExt} that
\begin{align}\label{Tfn_estimate}\left|\tau
    \left(T_{t,f^{[n]}}\bigl(x_1,\ldots,x_n\bigr) \right)\right|\leq
  \, c_n\, \left\|x_1\right\|_n \cdot \ldots \cdot
  \left\|x_n\right\|_n,\ \ 0 \leq t \leq 1,
\end{align} for every $x_1, \ldots, x_n\in L^n$ and for every~$f \in
\left( \cap_{k =0}^n W_k \right) \cap B$.  Combining~(\ref{f4}) and
(\ref{Tfn_estimate}) completes the proof.
\end{proof}


\begin{thebibliography}{99}
\bibitem{ADS06} N.~A.~Azamov, P.~G.~Dodds, F.~A.~Sukochev, \emph{The
Krein spectral shift function in semifinite von Neumann algebras,}
Integral Equations Operator Theory {\bf 55} (2006), 347 -- 362.

\bibitem{AzCaDoSu2009} N.~A.~Azamov, A.~L.~Carey, P.~G.~Dodds, and
F.~A.~Sukochev, \emph{Operator integrals, spectral shift, and
spectral flow}, Canad. J. Math. \textbf{61} (2009), no.~2, 241--263.

\bibitem{ACS} N.~A.~Azamov, A.~L.~Carey, F.~A.~Sukochev, {\it The
spectral shift function and spectral flow,} Comm. Math. Phys. {\bf
276} (2007), no.~1, 51--91.

\bibitem{BL} J.~Bergh, J.~L\"{o}fstr\"{o}m, \emph{Interpolation
spaces. An introduction}, Grundlehren der Mathematischen
Wissenschaften, vol. 223, Springer-Verlag, Berlin - New York, 1976.

\bibitem{BirmanKrein} M.~Sh.~Birman, M.~G.~Krein, {\it On the theory
of wave operators and scattering operators,}  Dokl. Akad. Nauk SSSR
{\bf 144} (1962), 475--478; English transl.\ in Soviet Math. Dokl.
{\bf 3} (1962), 740--744.

\bibitem{BirmanS}
M.~Sh.~Birman, M.~Z.~Solomyak, {\it Remarks on the spectral shift
function,} Zap. Nau\v cn. Sem. Leningrad. Otdel. Mat. Inst. Steklov.
(LOMI) {\bf 27} (1972), 33--46; English transl.\ in J. Soviet Math.
{\bf 3} (1975), no.~4, 408--419.

\bibitem{BirmanY}
M.~Sh.~Birman, D.~R.~Yafaev, {\it The spectral shift function. The
work of M.~G.~Krein and its further development,}  Algebra i Analiz
{\bf 4}  (1992),  no. 5, 1--44; English transl.\ in St. Petersburg Math.
J. {\bf 4} (1993), 833--870.

\bibitem{CareyPincus} R.~W.~Carey, J.~D.~Pincus, {\it Mosaics,
principal functions, and mean motion in von Neumann algebras,} Acta
Math. {\bf 138} (1977), no.~3-4, 153--218.

\bibitem{ChilinS} V.~I.~Chilin, F.~A.~Sukochev, \emph{Weak convergence
    in non-commutative symmetric spaces}, J. Operator Theory {\bf 31}
  (1994), no.~1, 35--65.

\bibitem{dVRLG1993} R.~A.~ DeVore and G.~G.~Lorentz,
\emph{Constructive approximation}, Grundlehren
  der Mathematischen Wissenschaften, vol. 303, Springer-Verlag, Berlin, 1993.

\bibitem{DDP} P.~G.~Dodds, T.~K.~Dodds, B.~de~Pagter, \emph{Fully
    symmetric operator spaces}, Integral Equations Operator Theory
  {\bf 15} (1992), no.~6, 942--972.

\bibitem{DDPS} P.~G.~Dodds, T.~K.~Dodds, B.~de~Pagter, F.~A.~Sukochev,
  \emph{Lipschitz continuity of the absolute value and Riesz
    projections in symmetric operator spaces}, J. Funct. Anal. {\bf
    148} (1997), no.~1, 28--69.

\bibitem{ds} K.~Dykema, A.~Skripka, {\it Higher order spectral
shift,} J. Funct. Anal., {\bf 257} (2009), 1092 -- 1132.

\bibitem{F1} Y.~B.~Farvorovskaya, \emph{An estimate of the nearness of the spectral decompositions
of self-adjoint operators in the Kantorovi\v{c}-Rubin\v{s}tein metric},
Vestnik Leningrad Univ. {\bf 22} (1967), no.~19, 155--156.

\bibitem{F2} Y.~B.~Farvorovskaya, \emph{The connection of the Kantorovi\v{c}-Rubin\v{s}tein
metric for spectral resolutions of self-adjoint operators with functions of operators},
Vestnik Leningrad Univ. {\bf 23} (1968), no.~19, 94--97.

\bibitem{F3} Y.~B.~Farvorovskaya, \emph{An example of a Lipschitz function of self-adjoint
operators with non-nuclear difference under a nuclear perturbation},
Zap. Nauchn. Sem. Leningrad. Otdel. Mat. Inst. Steklov. (LOMI) {\bf 30} (1972), 146--153.

\bibitem{GPS}
F.~Gesztesy, A.~Pushnitski, B.~Simon, {\it On the Koplienko
spectral shift function, I. Basics,} Zh. Mat. Fiz. Anal. Geom.
{\bf 4} (2008), no.~1, 63 -- 107.

\bibitem{Krein} M.~G.~Krein, \emph{On a trace formula in perturbation
theory}, Matem. Sbornik {\bf 33} (1953), 597 -- 626 (Russian).

\bibitem{Kop84} L.~S.~Koplienko, \emph{Trace formula for perturbations
    of nonnuclear type}, Sibirsk. Mat. Zh.  {\bf 25} (1984), 62-71
  (Russian). English transl.\ in Siberian Math. J. {\bf 25} (1984),
  735--743.

\bibitem{Lifshits} I.~M.~Lifshits, \emph{On a problem of the theory
of perturbations connected with quantum statistics,} Uspehi Matem.
Nauk {\bf 7} (1952), 171 -- 180 (Russian).

\bibitem{PSW2002}
B.~de~Pagter, F.~A.~Sukochev, H.~Witvliet, \emph{Double operator integrals},
J. Funct. Anal. {\bf 192} (2002) 52--111.

\bibitem{PellerKo} V.~V.~Peller, \emph{An extension of the
Koplienko-Neidhardt trace formulae,} J. Funct. Anal. {\bf 221}
(2005), 456--481.

\bibitem{PellerMult} V.~V.~Peller, \emph{Multiple operator integrals
and higher operator derivatives,} J. Funct. Anal. {\bf 223} (2006),
515--544.

\bibitem{PiXu2003} G.~Pisier and Q.~Xu, \emph{Non-commutative {$L\sp
p$}-spaces}, Handbook of the geometry of Banach spaces, Vol.\ 2, North-Holland, Amsterdam, 2003,
  pp.~1459--1517.

\bibitem{PS-Lipschitz} D.~Potapov and F.~Sukochev,
\emph{Operator-{L}ipschitz functions in {S}chatten-von {N}eumann
classes}, Acta Math., {\bf 207} (2011), 375 -- 389.

\bibitem{PS-Crelle} D.~Potapov and F.~Sukochev, \emph{Unbounded
Fredholm modules and double operator integrals}, J. reine. angew.
Math. {\bf 626} (2009), 159--185.

\bibitem{Schwartz1969} J.~T.~Schwartz, \emph{Nonlinear functional
analysis}, Gordon and Breach Science
  Publishers, New York, 1969.

\bibitem{unbdds} A.~Skripka, \emph{Higher order spectral shift, II.
Unbounded case}, Indiana Univ. Math. J. {\bf 59} (2010), no. 2, 691 -- 706.

\bibitem{moissf} A.~Skripka, {\it Multiple operator integrals and
spectral shift,} Illinois J. Math., to appear, arXiv:0907.0432.

\end{thebibliography}

\bibliographystyle{plain}

\end{document}